\documentclass{dhbenelux}

\usepackage[T1]{fontenc}
\usepackage{tgpagella}
\usepackage{tgheros}
\usepackage{microtype}
\usepackage{graphicx}
\usepackage{amsthm}
\usepackage{tikz-cd}

\usepackage{mathtools, mathrsfs}
\usepackage{amsthm} 
\usepackage{braket}
\usepackage[all]{xy}

\usepackage{enumitem}

\newtheorem{theorem}{Theorem}[section]
\newtheorem{proposition}[theorem]{Proposition}
\newtheorem{lemma}[theorem]{Lemma}

\newtheorem{remark}[theorem]{Remark}
\newtheorem{definition}[theorem]{Definition}
\newtheorem{example}[theorem]{Example}
\newtheorem{notation}[theorem]{Notation}
\newtheorem{corollary}[theorem]{Corollary}

\newtheorem*{theorem*}{Theorem}

\newcommand{\cA}{\mathcal{A}}
\newcommand{\cB}{\mathcal{B}}

\newcommand{\cF}{\mathcal{F}}
\newcommand{\cG}{\mathcal{G}}
\newcommand{\cH}{\mathcal{H}}

\newcommand{\cL}{\mathcal{L}}
\newcommand{\cM}{\mathcal{M}}
\newcommand{\cN}{\mathcal{N}}

\newcommand{\cS}{\mathcal{S}}
\newcommand{\cT}{\mathcal{T}}

\newcommand{\R}{\mathbb{R}}
\newcommand{\C}{\mathbb{C}}

\newcommand{\HP}[1]{\mathbb{H}P^{#1}}

\newcommand{\Hom}{\mathrm{Hom}}

\usepackage{hyperref}

\author{Christos T. Aravanis$^{\spadesuit}$ \and Ioannis S. Gkeneralis$^{\ast}$}

\title{Sheaves on quoric manifolds}
\date{}

\begin{document}
\maketitle

\begin{abstract}
Quoric manifolds are quaternionic analogues of quasitoric manifolds introduced by Hopkinson. Given a quoric manifold over a simple polytope, we study the stratification induced by the faces of the polytope. We then show that the equivariant derived category of constructible sheaves with respect to the face stratification is equivalent to the derived category of constructible DG-modules over a DG-algebra defined on the orbit space. Finally, we prove that the associated DG-algebra is formal.
\end{abstract}

\medskip

\noindent ${}^{\spadesuit}$Sheffield International College, University of Sheffield, England\\
\phantom{${}^{\bullet}$}Email: \href{mailto:c.aravanis@sheﬃeld.ac.uk}{\texttt{c.aravanis@sheﬃeld.ac.uk}}\medskip

\noindent ${}^{\ast}$Department of Mathematics, Aristotle University of Thessaloniki, Greece\\
\phantom{${}^{\bullet}$}Email: \href{mailto:igkeneralis@math.auth.gr}
{\texttt{igkeneralis@math.auth.gr}}\\

\noindent \textbf{Mathematics Subject Classification (2020):} 57S25, 32S60, 14F08, 55N91,  57M60\\

\noindent\textbf{Keywords:} quoric manifolds, equivariant sheaves, constructible sheaves, equivariant derived categories, finite spaces, DG-algebras, formality.

\tableofcontents
\section{Introduction} There were two motivations for this work. The first was to study quoric manifolds \cite{hopkinson2012quoric} from a sheaf-theoretic perspective and in particular, to understand how their equivariant topology is reflected in the associated categories of equivariant and constructible sheaves. The second was to investigate whether Lunts' techniques for toric varieties \cite{Lunts1995} can be applied to quoric manifolds. In particular, we wanted to determine whether the equivariant derived category of constructible sheaves on quoric manifolds admits a description in terms of the geometry of the orbit space and the isotropy data of the underlying group action.

In this paper we show that the equivariant derived category of constructible sheaves on a quoric manifold is equivalent to the derived category of constructible DG-modules over a sheaf of DG-algebras on the face poset of the orbit polytope. We further show that this sheaf of DG-algebras is formal, that is, quasi-isomorphic to its cohomology sheaf.

\subsection{Equivariant derived categories of sheaves}

Sheaf theory was introduced by Leray in the 1940s and developed in the following decade. Its basic principle is to organise local data on a space $X$ into global invariants: sheaves on $X$ form an abelian category $\mathrm{Sh}(X)$, and the global sections functor assigns to each sheaf its global invariants. This functor, and more generally the direct image $f_*$ along a continuous map $f$, is only left exact, and its failure to preserve short exact sequences is measured by higher direct images, in particular by sheaf cohomology. These higher terms are organised in the derived category, obtained from complexes of sheaves by localising at quasi-isomorphisms. There they become single functors $f_*$, which together with $f^*$, $\otimes^L$, $f_!$, $f^!$ and Verdier duality form the six operations \cite{kashiwara2013sheaves}.

This formalism is especially fruitful when $X$ carries extra structure. If $X$ is stratified, one restricts to complexes whose cohomology sheaves are constructible, that is, locally constant of finite rank along the strata. For suitable stratifications this subcategory is stable under the six operations and relates the local topology of the strata to the global structure of $X$. If a group $G$ acts on $X$, one obtains the abelian category $\mathrm{Sh}_G(X)$ of $G$-equivariant sheaves. Its naive derived category does not retain equivariant cohomology: for connected $G$ and $X=\mathrm{pt}$ it gives $D(\mathrm{Vect})$, whereas the equivariant cohomology of a point is $H^*(BG)$. This was addressed by Bernstein and Lunts \cite{Bernstein1994}, who define $D_G(X)$ by resolving $X$ by the free $G$-spaces $X\times E$ and descending to the quotients $X\times_G E$, where $E$ approximates $EG$. The resulting category carries the functors $f^*,f_*,f^!,f_!$ and the usual adjunctions. Combining the two structures gives the equivariant derived category $D_{G,c}(X)$ of constructible sheaves.

Toric varieties are the model case in which the group action, the stratification and the equivariant derived category are all controlled by combinatorics. A normal toric variety $X=X_\Sigma$ is a complex algebraic variety with an action of the algebraic torus $T^n:=(\C^\ast)^n$ and a dense open orbit. It has finitely many $T^n$-orbits $O_\sigma$, in bijection with the cones $\sigma$ of the fan $\Sigma$, with $\dim O_\sigma=n-\dim\sigma$. Lunts \cite{Lunts1995} described the bounded $T^n$-equivariant derived category $D^b_{T^n,c}(X)$ of sheaves constructible with respect to the orbit stratification algebraically. It is equivalent to the derived category of DG-modules over a differential graded algebra, which is the Ext-algebra of the direct sum of the simple equivariant perverse sheaves attached to the orbits, and which is formal. Lunts also conjectured that an analogous description holds for a complex reductive group acting on a projective variety with finitely many orbits; this conjecture was extended in a number of directions \cite{Guillermou05, Balthasar09, Schnurer2011, Joshua22}

In this paper we take a first step towards a real-topological analogue of this conjecture. We replace the complex reductive group by the compact Lie group $Q^n$, a product of copies of the group $Q\cong\mathrm{SU}(2)$ of unit quaternions, and the projective variety by a quoric manifold $M^{4n}$, the quaternionic analogue of a quasitoric manifold the latter being a topological analogue of a toric variety \cite{DJ91}. This is a $4n$-dimensional manifold with an action of $Q^n$ whose orbit space is a simple $n$-polytope $P^n$. As in the toric case, the faces of $P^n$ induce a finite stratification of $M^{4n}$. Our results are an equivalence of $D^b_{Q^n,c}(M^{4n})$ with the derived category of DG-modules over a sheaf of DG-algebras on the face poset of $P^n$ (Theorem~\ref{thm:main-theorem}), and the formality of this sheaf of DG-algebras (Theorem~\ref{prop:formality}). 

\subsection{Quoric manifolds, stratifications and the Borel construction}

Quoric manifolds were introduced by Hopkinson \cite{hopkinson2012quoric}, building on earlier work of Scott \cite{Scott95}, as quaternionic analogues of the quasitoric manifolds of Davis and Januszkiewicz \cite{DJ91}, which are themselves topological analogues of toric varieties \cite{fulton1993toric}. A quoric manifold is a smooth $4n$-dimensional manifold with a smooth locally regular action of $Q^n=(S^3)^n$ whose orbit space is a simple $n$-dimensional polytope (Definition~\ref{def:quoric}). This contrasts with a $2n$-dimensional quasitoric manifold, which carries an action of the compact torus $T^n=(S^1)^n$. The identification of the orbit space with the polytope is understood to preserve the face structure.

As in toric topology \cite{BuchstaberPanov2015}, the face structure of the orbit polytope records the isotropy data of the action (Section~\ref{sec:quoric-mnflds}). Since $Q^n$ is non-abelian, stabilisers are recorded up to conjugacy. Over the relative interior of each face their conjugacy class is constant, although the stabiliser subgroups need not be equal, and distinct faces may carry the same conjugacy class. Thus the partition by stabiliser conjugacy classes need not have connected pieces.

Let $\pi\colon M^{4n}\to P^n$ be the orbit map, let $\cF(P^n)$ be the set of nonempty faces of $P^n$ including $P^n$, and for a face $F$ with relative interior $\mathring F$ put $M_F:=\pi^{-1}(\mathring F)$. Then $M^{4n}=\bigsqcup_{F\in\cF(P^n)}M_F$. We encode this decomposition by a finite topological space: the face poset $(\cF(P^n),\preceq)$, ordered by inclusion, with the Alexandrov topology whose open sets are the up-sets (Section~\ref{sec:finite-spaces}). Denote it by $\cF_\preceq$. The map $\rho_P\colon P^n\to\cF_\preceq$, $\rho_P(x)=F$ for $x\in\mathring F$, is continuous and open, and so is
\begin{equation*}
\pi_{\cF}:=\rho_P\circ\pi\colon M^{4n}\longrightarrow\cF_\preceq,
\end{equation*}
whose fibres are precisely the subsets $M_F$. Our first main result is the following.

\begin{theorem*}[Theorem~\ref{thm:quoric-stratified}]
Let $M^{4n}$ be a quoric manifold over a simple polytope $P^n$. Then $\pi_{\cF}\colon M^{4n}\to\cF_\preceq$ defines a poset-stratified space, and the decomposition $M^{4n}=\bigsqcup_{F}\pi_{\cF}^{-1}(F)$ is a finite stratification of $M^{4n}$.
\end{theorem*}

The face stratification coincides with the connected components of the orbit-type Whitney stratification (Corollary~\ref{cor:face-orbit-type-comparison}) and extends to the Borel construction. Let $E_NQ^n=(S^{4N+3})^n$ and $EQ^n=\bigcup_NE_NQ^n$. In Section~\ref{sec:borel-construction} we study the Borel space construction of a  quoric manifold
\begin{equation*}
M_Q:=EQ^n\times_{Q^n}M^{4n}=\bigcup_NE_NQ^n\times_{Q^n}M^{4n}
\end{equation*}
together with the projection $b_Q\colon M_Q\to\cF_\preceq$, $[e,m]\mapsto\pi_{\cF}(m)$. For a face $F$ let $U_F=\{G\in\cF(P^n)\mid F\preceq G\}$ be its minimal open neighborhood and $V_{F,Q}:=b_Q^{-1}(U_F)$ the corresponding open star. Our second main result is the following.

\begin{theorem*}[Propositions~\ref{prop:loc-on-orbit-type-strata} and~\ref{prop:borel-deformation-retract}]
Let $F$ be a face of $P^n$ of codimension $k$ and $f\in\mathring F$. Then $EQ^n\times_{Q^n}M_F\simeq B\lambda^b(F)\cong BQ^k$, and the open star $V_{F,Q}$ deformation retracts onto $EQ^n\times_{Q^n}O_f\simeq B\lambda^b(F)$, where $O_f=\pi^{-1}(f)$ and $\lambda^b(F)\cong Q^k$ is the isotropy group of the stratum over $F$. In particular, both are connected and simply connected.
\end{theorem*}

This is the topological input for the sheaf of DG-algebras of Section~\ref{sec:quoric-dg-algebra}.

\subsection{Sheaves on quoric manifolds}
Sheaves that retain information about both the $Q^n$-action and the face stratification are our main interest. In Section~\ref{sect:sheaves-quoric} we study $Q^n$-equivariant sheaves on $M^{4n}$, with the face stratification $\{M_F\}_{F\in\cF(P^n)}$ as the organising structure. The aims are to compute the categories of sheaves and local systems on orbits and strata, and to use this to describe constructible and perverse sheaves.

Every stabiliser of the $Q^n$-action is connected, being isomorphic to $Q^k=\mathrm{SU}(2)^k$, and each stratum splits equivariantly as $M_F\cong\mathring F\times Q^n/\lambda^b(F)$. Since $\mathrm{SU}(2)$ is simply connected and $\mathring F$ is contractible, every stratum is simply connected. This gives two computations.
\begin{itemize}
\item \emph{Orbits.} The equivariant category of an orbit $Q^n/H$ is $\mathrm{Rep}(\pi_0(H))$, and $\pi_0(H)$ is trivial for every stabiliser. Hence $\mathrm{Sh}_{Q^n}(Q^n\cdot m)$, $\mathrm{Loc}(Q^n\cdot m)$ and $\mathrm{Loc}_{Q^n}(Q^n\cdot m)$ are all equivalent to $\mathrm{Vect}_{\C}$, so an equivariant sheaf on an orbit is constant.
\item \emph{Strata.} $\mathrm{Sh}_{Q^n}(M_F)\simeq\mathrm{Sh}(\mathring F)$, so the group action carries no extra data on a stratum, and $\mathrm{Loc}(M_F)\simeq\mathrm{Loc}_{Q^n}(M_F)\simeq\mathrm{Vect}_{\C}$. The proofs use the vanishing of $\pi_1$ of the stratum and of its Borel construction $EQ^n\times_{Q^n}M_F\simeq B\lambda^b(F)$.
\end{itemize}

All of the above have been well studied in general. Our aim is to carry out explicit calculations in the case of quoric manifolds. Our third main result is the following.

\begin{theorem*}[Theorem~\ref{thm:constructible-quoric}]
Let $M^{4n}$ be a quoric manifold with the natural face stratification $\{M_F\}_{F\in\cF(P^n)}$.
\begin{enumerate}
\item A sheaf $\cG$ is constructible if and only if, for every face $F$, the restriction $\cG|_{M_F}$ is a constant sheaf with fibre a finite-dimensional $\C$-vector space $V_F$.
\item The simple objects of $\mathrm{Perv}_{\mathcal S}(M^{4n})$, the perverse sheaves constructible with respect to $\mathcal S=\{M_F\}$, are, up to isomorphism, the complexes $\mathrm{IC}_F=\mathrm{IC}(M_F,\underline\C)$, one for each face $F$ of $P^n$, and they are pairwise non-isomorphic.
\end{enumerate}
\end{theorem*}

Thus constructibility on a quoric manifold involves no monodromy: a constructible sheaf is determined by the vector spaces $V_F$ and the gluing of the strata, and the simple perverse sheaves are parametrised by the faces of the polytope. The section also recalls the background on sheaves, equivariant sheaves, local systems, recollement and $t$-structures used in the proofs.

\subsection{The equivariant derived category of a quoric manifold}
The study of the equivariant derived category of a quoric manifold can be approached through the combinatorial structure of its orbit space. The face poset $\cF_\preceq$ provides a finite model for organising the local equivariant data, while a sheaf of DG-algebras on this poset encodes the corresponding algebraic information.

By the Borel deformation retraction (Proposition~\ref{prop:borel-deformation-retract}), the local equivariant data over $U_F$ are those of $B\lambda^b(F)$, so they are organised by the finite space $\cF_\preceq$. In Section~\ref{sec:quoric-dg-algebra} we associate to $b_Q$ the sheaf of DG-algebras $\cA=b_{Q,*}\Omega^\bullet_{M_Q}$, where $\Omega^\bullet_{M_Q}$ is the complex-valued de Rham sheaf of $M_Q$, defined as the limit of the de Rham complexes of the approximations $M_Q^N=E_NQ^n\times_{Q^n}M^{4n}$. Its stalks are computed in Proposition~\ref{prop:stalks-of-cA}: $H^*(\cA_F)\cong\C[x_1,\dots,x_k]$ with $k=\mathrm{codim}\,F$ and $\deg x_i=4$. A sheaf on the finite space $\cF_\preceq$ is determined by finitely many stalks and restriction maps. The modules $\cA_{U_F}$ form a finite collection of DG-$\cA$-modules which generates the constructible derived category $D_{\cA,c}$ (Lemma~\ref{lem:generation-by-minimal-opens}).

Our fourth main result is the algebraic description of the $Q^n$-equivariant derived category of constructible sheaves on a quoric manifold. 
\begin{theorem*}[Theorem~\ref{thm:main-theorem}]
Let $M^{4n}$ be a quoric manifold over a simple polytope $P^n$, and let $\cA=b_{Q,*}\Omega^\bullet_{M_Q}$ be the associated sheaf of DG-algebras on $\cF_\preceq$. Then the functors $\Gamma$ and $L$ restrict to mutually quasi-inverse equivalences of triangulated categories
\begin{equation*}
\Gamma\colon D^b_{Q^n,c}(M^{4n})\ \xrightarrow{\ \sim\ }\ D_{\cA,c},
\end{equation*}
where $D_{\cA,c}$ is the derived category of constructible DG-$\cA$-modules.
\end{theorem*}

We outline the proof, which follows the strategy of Lunts \cite{Lunts1995}. The functors
\begin{equation*}
\Gamma\colon D^+(M_Q)\to D_{\cA},\qquad L\colon D_{\cA}\to D(M_Q),
\end{equation*}
are given by $\Gamma(\cG^\bullet)=b_{Q,*}(\Omega^\bullet_{M_Q}\otimes_{\C}\cG^\bullet)$ and $L(\cN)=\Omega^\bullet_{M_Q}\otimes_{b_Q^{-1}\cA}b_Q^{-1}R(\cN)$, where $R(\cN)\to\cN$ is a projective resolution. They are defined on the full categories, and the theorem asserts that they are inverse to each other on the constructible subcategories.

The proof reduces to a finite collection of generators. On the DG-module side, $D_{\cA,c}$ is generated by $\{\cA_{U_F}\}_{F\in\cF(P^n)}$. On the equivariant side, let $j_F\colon V_{F,Q}\hookrightarrow M_Q$ be the open embedding of the open star. Then $\{j_{F!}\underline{\C}_{V_{F,Q}}\}_{F\in\cF(P^n)}$ generates $D^b_{Q^n,c}(M^{4n})$ (Lemma~\ref{lem:generation-sheaf-side}). We show that the two functors exchange these families, namely $\cA_{U_F}\simeq\Gamma(j_{F!}\underline{\C}_{V_{F,Q}})$ and $j_{F!}\underline{\C}_{V_{F,Q}}\simeq L(\cA_{U_F})$. Finally, we construct natural transformations $u\colon\mathrm{Id}\to\Gamma L$ and $c\colon L\Gamma\to\mathrm{Id}$ and show that they are isomorphisms on the generators. The full subcategories on which they are isomorphisms are triangulated, so  are isomorphisms on all of $D_{\cA,c}$ and $D^b_{Q^n,c}(M^{4n})$.

\subsection{Formality of the DG-algebra}

The fifth and last main result of this paper concerns the sheaf of DG-algebras $\cA$ itself. By Proposition~\ref{prop:stalks-of-cA}, the cohomology of each stalk $\cA_F$ is a polynomial ring on generators of degree $4$. Motivated by \cite{Lunts1995} we show that $\cA$ is quasi-isomorphic to its cohomology. In particular we have the following.

\begin{theorem*}[Theorem~\ref{prop:formality}]
Let $M^{4n}$ be a quoric manifold over a simple polytope $P^n$. The sheaf of commutative DG-algebras $\cA=b_{Q,*}\Omega^\bullet_{M_Q}$ on $\cF_\preceq$ is formal, that is, quasi-isomorphic to its cohomology sheaf $\cH=H^\ast(\cA)$ with zero differential.
\end{theorem*}

 The steps of the proof are outlined below:
\begin{enumerate}
\item The restrictions of the closed global $4$-forms $\Omega_i=\pi^*\omega_i$, which represent the generators of $H^\ast(BQ^n;\R)$, generate the cohomology of every stalk (Proposition~\ref{prop:generation}). This uses Hopkinson's description of the subgroups of $Q^n$ isomorphic to $Q^k$.
\item Using the vanishing of $H^1,H^2,H^3$ of the stalks, we choose in each $\cA_F$ a subcomplex $C_F$, with no choices involved and hence compatible with restriction, whose cohomology is concentrated in degree $4$ (Lemma~\ref{lem:cohomology-CF}). In Lunts's toric case the generators have degree $2$, so only $H^1$ is needed.
\item The free graded-commutative algebras $\cB(U_F)=\operatorname{Sym}(C_F)$ form a sheaf $\cB$, and the maps $\cB\to\cA$ and $\cB\to\cH$ are quasi-isomorphisms, checked on the stalks.
\end{enumerate}

Combined with Theorem~\ref{thm:main-theorem}, formality gives a purely algebraic model: $\cA$ may be replaced by the sheaf of graded algebras $\cH$, which is determined by the isotropy data of the action.
\subsection{Future directions and applications}

We conclude with possible applications to other spaces with a locally standard torus action and a face stratification. For a quasitoric manifold $X^{2n}$ over a simple polytope $P^n$ \cite{DJ91}, the isotropy group of a face $F$ of codimension $k$ is a compact torus $T_F\cong T^k$ with $BT_F\simeq(\mathbb CP^\infty)^k$. Replacing $(\mathbb HP^\infty)^k$ by $(\mathbb CP^\infty)^k$, one anticipates the same finite-space and DG-algebraic description, without any need for a quoric structure. Quaternionic moment-angle manifolds \cite{hopkinson2012quoric,gken_rigidityCH_moment-angle} form a closely related class, with a face stratification and isotropy groups of the same type, and are linked to the quaternionic Davis--Januszkiewicz space $DJ_{\mathbb H}(K)=(\mathbb HP^\infty,*)^K$, where $K$ is the simplicial complex dual to the boundary of $P^n$. Whether the corresponding DG-algebraic equivalence holds for them has not been verified.

\subsection{Organisation of the paper}

Section~\ref{sec:quoric-stuff} treats quoric manifolds, their face stratification and the Borel construction. Section~\ref{sect:sheaves-quoric} studies equivariant and constructible sheaves on them. Section~\ref{sec:main-section} contains the main equivalence, Theorem~\ref{thm:main-theorem}, and Section~\ref{sec:formality} proves the formality of $\cA$.

\subsection*{Acknowledgments} This project was initiated during the 16th Panhellenic Geometry Conference, held at the University of Athens in 2024. The authors are grateful to P.~Batakidis and S.~Prassidis for their comments, questions and suggestions in the early stages of this work. We also thank  I.~Emmanouil and P.~Gianniotis for providing office space at the Department of Mathematics of NKUA during several visits. The first named author is especially grateful to S.~Guillermou and J.~Woolf for their patience in answering his questions, and to I.~Androulidakis for generously sharing his insights. 

What little we understand about equivariant derived categories owes much to the foundational work of V.~Lunts~\cite{Lunts1995}, to whom we feel indebted; without his work, this paper would not have been possible.

\section{Quoric manifolds, finite spaces and stratifications}
\label{sec:quoric-stuff}

\subsection{Quaternionic actions}\label{sec:quaternionic-actions}

Let $\mathbb H$ be the associative algebra of quaternions and denote by $Q$ the quaternions of unit length, namely $Q:=\{q\in\mathbb H:\|q\|=1\}$.  The set $Q$ with the multiplication of quaternions forms a compact, non-commutative, real algebraic and reductive Lie group which is isomorphic to $S^3$ and hence with $\mathrm{SU}(2)$; symbolically
\begin{equation*}
Q\cong S^{3}\cong \mathrm{SU}(2).
\end{equation*}

\begin{notation}
In what follows, we will denote by $Q^n$ the $n$-fold product. 
\end{notation}
 
\noindent  The \emph{standard coordinatewise action} of $Q^{n}$ on $\mathbb H^{n}$ is defined by
\begin{equation*}
(q_{1},\ldots,q_{n})\cdot(h_{1},\ldots,h_{n})=(q_{1}h_{1},\ldots,q_{n}h_{n}).
\end{equation*} 
For a point $h=(h_{1},\ldots,h_{n})\in \mathbb H^{n}$, we denote its isotropy subgroup by
\begin{equation*}
Q^{n}_{h} = \{(q_{1},\ldots,q_{n})\in Q^{n} : q_{i}h_{i}=h_{i} \text{ for all }i\}.
\end{equation*} 
Since the isotropy subgroup of a nonzero quaternion under left multiplication is trivial, whereas the origin is fixed by the whole group $Q$, we obtain $Q_h^n\cong Q^k$, where $k=|\{i\in[n]:h_i=0\}|$; see \cite[\S 2.2, pp.~27--28]{hopkinson2012quoric}.

Consequently, the isotropy subgroup of a point $h\in\mathbb H^n$ is determined by the set of its vanishing coordinates. In particular, points with all coordinates nonzero have trivial isotropy and belong to the principal orbit type of the standard $Q^n$-action.

For a general action of a Lie group $G$ on a space $X$, the isotropy subgroups of points belonging to the same orbit are conjugate. Indeed, for $g\in G$ and $x\in X$, one has $G_{g\cdot x}=gG_xg^{-1}$. Thus, the conjugacy class $(G_x):=\{gG_xg^{-1}:g\in G\}$ is naturally associated with the orbit $G\cdot x$. 

\begin{notation}
Throughout this paper we shall refer to such conjugacy classes simply as \emph{isotropy classes}.
\end{notation}

As explained in \cite[\S 2.2, pp.~27--28]{hopkinson2012quoric}, the orbit space of the standard $Q^n$-action on $\mathbb H^n$ is naturally identified with the positive orthant $\mathbb R_{\geq0}^n$ through the map
\begin{equation*}
\mu:\mathbb H^n\longrightarrow\mathbb R_{\geq0}^n,\qquad\mu(h_1,\ldots,h_n)=(|h_1|,\ldots,|h_n|).
\end{equation*}
Under this identification, the relative interior of each face of $\mathbb R_{\geq0}^n$ consists of points with a prescribed set of vanishing coordinates. Consequently, the isotropy subgroup is constant over the preimage of each such relative interior.

The faces of $\mathbb R_{\geq0}^n$ are determined by the vanishing of coordinate functions. More precisely, for each $i\in[n]:=\{1,\ldots,n\}$, let $F_i=\{(r_1,\ldots,r_n)\in\mathbb R_{\geq0}^n:r_i=0\}$ be the corresponding facet. Every face can be written as
\begin{equation*}
F_\sigma=\bigcap_{i\in\sigma}F_i,
\qquad \sigma\subseteq[n],
\end{equation*}
with the convention $F_\emptyset=\mathbb R_{\geq0}^n$. In particular,
\begin{equation*}
\mathrm{codim}F_\sigma=|\sigma|,
\qquad
F_\sigma\subseteq F_\tau
\Longleftrightarrow
\tau\subseteq\sigma.
\end{equation*}

Thus, the faces of $\mathbb R_{\geq0}^n$, ordered by inclusion, form a poset. Moreover, the isotropy subgroup associated with the relative interior of $F_\sigma$ is the coordinate subgroup
\begin{equation*}
Q_\sigma=\{(q_1,\ldots,q_n)\in Q^n:q_i=1 \text{ for }i\notin\sigma\}\cong Q^{|\sigma|}.
\end{equation*}
In particular, the isotropy subgroup associated with a face of codimension $k$ is isomorphic to $Q^k$; see \cite[\S 2.2, p.~28]{hopkinson2012quoric}.

Thus, the face structure of the positive orthant $\mathbb R_{\geq0}^n$ encodes the isotropy subgroups of the standard $Q^n$-action. Passing to conjugacy classes, this correspondence associates to each face $F_\sigma$ the isotropy class $(Q_\sigma)$. Moreover, if $F_\sigma\subseteq F_\tau$, then $Q_\tau\subseteq Q_\sigma$.

Let $\cF(\mathbb R_{\geq0}^n)$ denote the category whose objects are the faces of $\mathbb R_{\geq0}^n$ and whose morphisms are inclusions. Since
\begin{equation*}
F_\sigma\subseteq F_\tau
\quad\Longleftrightarrow\quad
\tau\subseteq\sigma,
\end{equation*}
there is a canonical isomorphism of categories $\cF(\mathbb R_{\geq0}^n)^{\mathrm{op}}\cong \mathrm{cat}([n])$, where $\mathrm{cat}([n])$ is the category of subsets of $[n]$ ordered by inclusion.

Following \cite[Definition~3.1.18]{hopkinson2012quoric}, we have the following definition:

\begin{definition}\label{def:acceptability}
Let $\mathrm{conj}(Q^n)$ denote the category whose objects are conjugacy classes of closed subgroups of $Q^n$ isomorphic to $Q^k$, for some $0\leq k\leq n$, and whose morphisms are determined by the subclass relation. An \emph{isotropy functor} is a functor
\begin{equation*}
\ell:\mathrm{cat}([n])
\longrightarrow \mathrm{conj}(Q^n).
\end{equation*}
It is called \emph{acceptable} if it is injective on objects.
\end{definition}

\begin{remark}
\label{rem:acceptable-isotropy}
Recall that $\mathrm{cat}([n])\cong\cF(\mathbb R_{\geq0}^n)^{\mathrm{op}}$. Thus, an acceptable isotropy functor assigns to each face $F_\sigma$ of the positive orthant an isotropy class $\ell(\sigma)$, in a manner compatible with the reverse inclusion of faces.

Moreover, by
\cite[Proposition~3.1.19]{hopkinson2012quoric}, if $F_\sigma$ has codimension $k=|\sigma|$, then the isotropy class $\ell(\sigma)$ has rank $k$. In particular, the rank of the isotropy class increases as one passes to faces of higher codimension.
\end{remark}

\begin{remark}
It should be mentioned that Hopkinson gives several equivalent characterizations of acceptability, including formulations in terms of matrices and graphs \cite[p.60]{hopkinson2012quoric}. We shall not need these descriptions here, and therefore use only the isotropy-functor formulation.
\end{remark}

The isotropy functor describes the local orbit structure of the standard $Q^n$-action. The corresponding global object for a quoric manifold is the characteristic functor, which we introduce in the next subsection.

\begin{remark}
Hopkinson allows also more general actions of $Q^n$ on $\mathbb H^n$ of the form

\begin{equation} \label{eq:general_action}
((s_1,...,s_n),(h_1,...,h_n))\mapsto (s_{l_1}h_1s_{r_1}^{-1},...,s_{l_n}h_ns_{r_n}^{-1})
\end{equation}  
\noindent
where each $l_i$ and $r_i$ is either a coordinate index or the empty set (with $s_\emptyset=1$). Notice  that for $l_i=i$ and $r_i=\emptyset$ for all $i$ recovers the coordinatewise action
\begin{equation*}
(s_1,\ldots,s_n)\cdot(h_1,\ldots,h_n)=(s_1h_1,\ldots,s_nh_n).
\end{equation*}
\end{remark}

However, arbitrary actions of the form \eqref{eq:general_action} need not share the orbit-space structure of the standard coordinatewise action.

We next explain how an isotropy functor arises from a group action. Suppose that a $Q^n$-action has orbit space identified with $\mathbb R_{\geq0}^n$ and that its isotropy class is constant over the preimage of the relative interior of each face. Then the action determines an isotropy assignment
\begin{equation*}
\ell(\sigma)=(Q^n_h),
\qquad
\mu(h)\in\mathrm{relint}(F_\sigma),
\end{equation*}
where $\mu$ denotes the orbit projection. When this assignment preserves the subclass relations, it defines an isotropy functor
\begin{equation*}
\ell:\mathrm{cat}([n])\longrightarrow\mathrm{conj}(Q^n).
\end{equation*}

\begin{definition}
A smooth action of $Q^n$ on $\mathbb H^n$ admitting an isotropy functor as described above is called a \emph{regular action} if its associated isotropy functor is acceptable.
\end{definition}

\begin{definition}\label{def:regular_corner}
The space $\mathbb H^n$ equipped with a regular action is called a \emph{regular $Q^n$-corner}.
\end{definition}

\subsection{Quoric manifolds}\label{sec:quoric-mnflds}

Here we recall the definition of quoric manifolds, following Hopkinson \cite{hopkinson2012quoric}, and introduce the geometric and combinatorial structures needed in the sequel. 

\begin{definition} [cf.~{\cite[Definition~4.2.1]{hopkinson2012quoric}}] \label{def:loc_reg}
Let $Q^n$ act smoothly on a smooth $4n$-manifold $M^{4n}$. A \emph{regular chart} is a pair $(U,\varphi)$, where $U\subseteq M^{4n}$ is a $Q^n$-invariant open subset and $\varphi\colon U\to\mathbb{H}^n$ is an equivariant diffeomorphism onto a regular $n$-corner, that is, $\varphi(t\cdot y)=t\cdot\varphi(y)$ for all $t\in Q^n$, $y\in U$, where $Q^n$ acts on $\mathbb{H}^n$ as in \cite[Definition 3.1.22]{hopkinson2012quoric}. The action is \emph{locally regular} if every point of $M^{4n}$ lies in a regular chart and the transition maps between overlapping charts are smooth.
\end{definition}

\begin{remark}
If $(U,\varphi)$ is a regular chart, then equivariance and injectivity of $\varphi$ give $(Q^n)_x=(Q^n)_{\varphi(x)}$ for all $x\in U$, so the isotropy subgroups of the action on $M^{4n}$ are those of the regular corner.
\end{remark}

We now define quoric manifolds.

\begin{definition}[cf. {\cite[Definition 4.2.2]{hopkinson2012quoric}}]\label{def:quoric}
Let $P^n$ be a simple polytope. A \emph{quoric manifold} over $P^n$ is a $4n$-dimensional manifold $M^{4n}$ equipped with a $Q^n$-action such that:
\begin{itemize}
\item[(1)] the action of $Q^n$ on $M^{4n}$ is locally regular (in the sense of Definition \ref{def:loc_reg});
\item[(2)] there is an orbit space projection map
\begin{equation*}\label{def:orbit-map}
\pi:M^{4n}\to P^n
\end{equation*} that is its fibres are  the $Q^n$-orbits. 
\end{itemize}
\end{definition}

\begin{remark}\label{rem:orbit-space-diff-polytope}
Condition (2) of Definition~\ref{def:quoric} says that $\pi$ is an orbit space projection, that is, it identifies the orbit space $M^{4n}/Q^n$ with $P^n$. Explicitly, if $p\colon M^{4n}\to M^{4n}/Q^n$ denotes the orbit map, then $\pi=\bar\pi\circ p$ for a homeomorphism $\bar\pi\colon M^{4n}/Q^n\to P^n$, so the following triangle commutes:
\begin{equation*}
\begin{tikzcd}[row sep=huge, column sep=huge]
  M^{4n} \arrow[r, "\pi"] \arrow[d, "p"'] & P^n \\
  M^{4n}/Q^n \arrow[ur, "\bar\pi"', "\cong"] &
\end{tikzcd}
\end{equation*}
\end{remark}

The prototypical example of a quoric manifold is the following.

\begin{example}\label{projective_quoric}
Let $\mathbb HP^n$ be the quaternionic projective space, the quotient of $\mathbb H^{n+1}\setminus\{0\}$ by right multiplication by $\mathbb H^*$. The action of $Q^n$ by left multiplication on the coordinates,
\begin{equation*}
(q_1,\dots,q_n)\cdot[h_0:h_1:\dots:h_n]=[h_0:q_1h_1:\dots:q_nh_n],
\end{equation*}
makes $\mathbb HP^n$ a $4n$-dimensional quoric manifold over the $n$-simplex
$\Delta^n$, with orbit space projection
\begin{equation*}
\pi([h_0:\dots:h_n])=(|h_0|^2,\dots,|h_n|^2).
\end{equation*}
See \cite[Example 4.2.5]{hopkinson2012quoric} for details.
\end{example}

As discussed in the preceding subsection, the orbit space of the standard coordinatewise $Q^n$-action on $\mathbb H^n$ is naturally identified with the positive orthant $\mathbb R_{\geq0}^n$. More generally, the orbit space of a regular $Q^n$-corner has the same local manifold-with-corners structure; see \cite[\S 3.1]{hopkinson2012quoric}. Consequently, if a $Q^n$-action on a smooth $4n$-manifold $M$ is locally regular, its orbit space $M/Q^n$ is locally modelled on open subsets of $\mathbb R_{\geq0}^n$ and hence carries the structure of an $n$-dimensional topological manifold with corners.

In the case of a quoric manifold over a simple polytope $P^n$, the orbit space is identified with $P^n$. We now recall the combinatorial structure of its faces, following \cite[p.74]{hopkinson2012quoric}.

Let $F_1,\ldots,F_m$ denote the facets of $P^n$. Since $P^n$ is simple, each face of codimension $k$ is the intersection of precisely $k$ facets containing it. Thus, every face can be written uniquely as
\begin{equation*}
F_\sigma=\bigcap_{i\in\sigma}F_i,
\qquad \sigma\subseteq[m],
\end{equation*}
provided that the intersection is nonempty, with $F_\emptyset=P^n$.

The faces of $P^n$, ordered by inclusion, form the \emph{face poset} $\cF(P^n)$. We denote by $\cF(P^n)$ the corresponding category, whose morphisms are face inclusions. The face poset of $P^n$ determines the orbit types of a locally regular $Q^n$-action through the isotropy assignment.

\begin{definition}\label{def:char_function}
A \emph{characteristic functor} on a simple polytope $P^n$ is a functor
\begin{equation*}
\lambda:
\cF(P^n)^{\mathrm{op}}
\longrightarrow
\mathrm{conj}(Q^n)
\end{equation*}
such that, for every vertex $v$ of $P^n$, its restriction to the faces containing $v$ is an acceptable isotropy functor, under the natural identification of the corresponding opposite face category with $\mathrm{cat}([n])$. The pair $(P^n,\lambda)$ is called a \emph{characteristic pair}.
\end{definition}

\begin{remark}
The acceptability condition for a characteristic functor $\lambda$ is local: for each vertex $v$ of $P^n$, the restriction of $\lambda$ to the faces containing $v$ is injective on objects. In particular, $\lambda$ need not be globally injective, since distinct faces not belonging to a common vertex neighbourhood may have the same isotropy class. Moreover, no surjectivity condition is imposed on $\lambda$; not every conjugacy class of subgroups of $Q^n$ need occur as the image of a face.
\end{remark}

Given a characteristic pair $(P^n,\lambda)$, one may construct a quoric manifold. To be precise, for each point $x\in P^n$, let $F_x$ denote the unique face whose relative interior contains $x$. Following \cite[Construction~4.2.8]{hopkinson2012quoric}, let $\lambda^b(F_x)<Q^n$ denote the canonical subgroup representing the isotropy class $\lambda(F_x)$. We use the compatible subgroup representatives $\lambda^b(F)\leq Q^n$ provided by the canonical construction, so that
\begin{equation*}
F\subseteq G
\quad\Longrightarrow\quad
\lambda^b(G)\subseteq\lambda^b(F).
\end{equation*}
Define an equivalence relation $\sim_\lambda$ on $Q^n\times P^n$ by
\begin{equation*}
(q,x)\sim_\lambda(q',x')
\quad\Longleftrightarrow\quad
x=x'
\text{ and }
q^{-1}q'\in\lambda^b(F_x).
\end{equation*}

\begin{definition}\label{def:canonical-model}
The quotient space
\begin{equation*}
\mathcal M(\lambda):=(Q^n\times P^n)/\!\sim_\lambda
\end{equation*}
is called the \emph{canonical model}, or \emph{derived space}, associated to the characteristic pair $(P^n,\lambda)$.
\end{definition}

The derived space $\cM(\lambda)$ is a quoric manifold with a $Q^n$-action defined by left multiplication on the first factor, while leaving the second factor unchanged; see \cite[Proposition 4.2.9]{hopkinson2012quoric}. Its orbit space $\cM(\lambda)/Q^n$ is naturally identified with $P^n$. Moreover, for a quoric manifold $M^{4n}$ over a simple polytope $P^n$ and a characteristic functor $\lambda$, there is a $Q^n$-equivariant diffeomorphism between $M^{4n}$ and its derived space $\mathcal{M}(\lambda)$; see \cite[Proposition 4.2.10]{hopkinson2012quoric}.

In the next lemma, we use the canonical model to prove an equivariant deformation retraction result.

\begin{lemma}[Equivariant deformation retraction]
\label{lemma:deformation-retract} Let $F$ be a face of the simple polytope $P^n$, and define its star by $U_F:=\bigcup_{G\succeq F}\mathring G\subseteq P^n$. Set $V_F:=\pi^{-1}(U_F)\subseteq M^{4n}$. Fix $f\in\mathring F$ and write
\begin{equation*}
O_f:=\pi^{-1}(f)\cong Q^n/\lambda^b(F).
\end{equation*}
There exists a $Q^n$-equivariant homotopy
\begin{equation*}
H\colon[0,1]\times V_F\longrightarrow V_F
\end{equation*}
such that
\begin{equation*}
H(1,\cdot)=\mathrm{id}_{V_F},
\qquad H(0,V_F)=O_f,
\qquad H(t,\cdot)|_{O_f}=\mathrm{id}_{O_f}
\end{equation*}
for all $t\in[0,1]$. Moreover, for every face $G\succeq F$,
\begin{equation*}
H_t\bigl(\overline{M_G}\cap V_F\bigr)
\subseteq \overline{M_G}\cap V_F
\qquad (t\in[0,1]),
\end{equation*}
where $M_G:=\pi^{-1}(\mathring G)$. 
\end{lemma}
\begin{proof}
Recall from Definition~\ref{def:canonical-model} that
\begin{equation*}
\mathcal M(\lambda)=(Q^n\times P^n)/\!\sim,
\end{equation*}
where $(q,x)\sim(q',x)$ if $x\in\mathring K$ for some face $K$ and $q^{-1}q'\in\lambda^b(K)$. Since $M^{4n}$ is diffeomorphic to $\mathcal M(\lambda)$, we identify these spaces via the given diffeomorphism. Let
\begin{equation*}
p\colon Q^n\times P^n\longrightarrow\mathcal M(\lambda),
\qquad p(q,x)=[q,x],
\end{equation*}
denote the quotient projection. Under this identification, the $Q^n$-action is given by $s\cdot[q,x]=[sq,x]$, and $\pi([q,x])=x$. Define
\begin{equation*}
h\colon[0,1]\times U_F\longrightarrow U_F,
\qquad h(t,x)=tx+(1-t)f.
\end{equation*}
The set $U_F$ is convex and contains $f$, so $h$ is well-defined. It satisfies
\begin{equation*}
h(1,x)=x,\qquad h(0,x)=f,\qquad h(t,f)=f.
\end{equation*}
Furthermore, let $x\in\mathring K$ with $K\succeq F$, so $f\in F\subseteq K$. Write $P^n=\{\ell_i\ge0\}_{i\in[m]}$ for the facet inequalities. If $K\subseteq F_i$, then $\ell_i(x)=\ell_i(f)=0$, so $\ell_i(h(t,x))=0$. If $K\not\subseteq F_i$, then $\ell_i(x)>0$ and $\ell_i(f)\ge0$, so $\ell_i(h(t,x))=t\ell_i(x)+(1-t)\ell_i(f)>0$ for $t>0$. Thus $h(t,x)$ lies on exactly the same facets as $x$, that is, $h(t,x)\in\mathring K$ for all $t>0$.

Now, we define
\begin{equation*}
H(t,[q,x]):=[q,h(t,x)].
\end{equation*}
For well-definedness, let $[q,x]=[q',x']$. Then $x=x'\in\mathring K$ for some face $K\succeq F$, and $q^{-1}q'\in\lambda^b(K)$. If $t>0$, then $h(t,x)\in\mathring K$, so
\begin{equation*}
[q,h(t,x)]=[q',h(t,x)].
\end{equation*}
If $t=0$, then $h(0,x)=f\in\mathring F$. The compatibility assumption
\begin{equation*}
\lambda^b(K)\subseteq\lambda^b(F)
\qquad(F\preceq K)
\end{equation*}
therefore gives $[q,f]=[q',f]$. Hence $H$ is well-defined.

Next, we show the $Q^n$-equivariance.  For $s\in Q^n$, we have
\begin{equation*}
H(t,s\cdot[q,x]) =H(t,[sq,x])=[sq,h(t,x)]=s\cdot[q,h(t,x)]=s\cdot H(t,[q,x]).
\end{equation*}

For continuity, set $p_F:=p|_{Q^n\times U_F}\colon Q^n\times U_F\to V_F$. Since $U_F$ is open in $P^n$ and $\sim_\lambda$ only identifies points with the same $P^n$-coordinate, $Q^n\times U_F$ is an open saturated subset of $Q^n\times P^n$, so $p_F$ is a quotient map. As $[0,1]$ is locally compact Hausdorff, $\mathrm{id}_{[0,1]}\times p_F$ is a quotient map as well. The map
\begin{equation*}
\widetilde H\colon[0,1]\times Q^n\times U_F\longrightarrow V_F,\qquad
\widetilde H(t,q,x)=p(q,h(t,x)),
\end{equation*}
is continuous, and it is constant on the fibres of $\mathrm{id}_{[0,1]}\times p_F$ by the well-definedness argument above. Hence it descends to a continuous map $H\colon[0,1]\times V_F\to V_F$.

Finally, at the endpoints, we have
\begin{equation*}
H(1,[q,x])=[q,x], \qquad H(0,[q,x])=[q,f].
\end{equation*}
and as a result,
\begin{equation*}
H(1,\cdot)=\mathrm{id}_{V_F}, \qquad H(0,V_F)=\pi^{-1}(f)=O_f.
\end{equation*}
Since $h(t,f)=f$, the homotopy fixes $O_f$ pointwise for all $t\in[0,1]$. This proves that $H$ is a $Q^n$-equivariant deformation retraction. Finally, let $G\succeq F$. We claim that
\begin{equation*}
\overline{M_G}\cap V_F=\overline{\pi^{-1}(\mathring{G})}\cap V_F = \pi^{-1}(G)\cap V_F.
\end{equation*}
Indeed, since $\pi$ is continuous and $G$ is closed in $P^n$, we have $\overline{M_G}\subseteq\pi^{-1}(G)$. Conversely, using $G=\overline{\mathring G}$, we obtain
\begin{equation*}
\pi^{-1}(G) =p(Q^n\times G) \subseteq\overline{p(Q^n\times\mathring G)} =\overline{M_G}.
\end{equation*}
Thus, the claimed equality follows.

Now, let $[q,x]\in\overline{M_G}\cap V_F$. Then $x\in G$, and since $f\in F\subseteq G$, the convexity of $G$ implies that
\begin{equation*}
h(t,x)=tx+(1-t)f\in G \qquad\text{for all }t\in[0,1].
\end{equation*}
Moreover, $h(t,x)\in U_F$. Hence
\begin{equation*}
H(t,[q,x]) =[q,h(t,x)] \in\pi^{-1}(G)\cap V_F =\overline{M_G}\cap V_F.
\end{equation*}
It follows that
\begin{equation*}
H_t\bigl(\overline{M_G}\cap V_F\bigr) \subseteq\overline{M_G}\cap V_F
\end{equation*}
for every face $G\succeq F$ and every $t\in[0,1]$. This completes the proof.
\end{proof}

\subsection{Finite spaces}\label{sec:finite-spaces}

The orbit space of a quoric manifold $M^{4n}$ is a simple polytope $P^n$. The set of faces $\cF(P^n)$ with partial order the face inclusion is a poset. As we explain in this section, the face poset $(\cF(P^n),\preceq)$ naturally gives rise to a finite topological space. Finite topological spaces will play a fundamental role in Section~\ref{sec:main-section} where we describe the equivariant derived category of constructible sheaves on a quoric manifold. We start with the necessary definitions from \cite{MayFiniteSpaces2008}.

\begin{definition}
Let $(X,\tau)$ be a topological space. Then, $X$ will be called 
\begin{itemize}
    \item[(1)] a \textbf{finite topological space} if the underlying set $X$ is finite;
    \item[(2)] an \textbf{Alexandrov  topological space} if $\tau$ is closed under arbitrary intersections.
\end{itemize}
\end{definition}

Every finite topological space is an Alexandrov space, see \cite[Lemma~1.3]{MayFiniteSpaces2008}. More precisely, there is a bijective correspondence between preorders on a finite set $X$ and Alexandrov topologies on $X$, under which a preorder $\le$ corresponds to the topology whose minimal open neighbourhood of each point $x\in X$ is
\begin{equation*}
U_x = \{\, y \in X \mid x \le y \,\}.
\end{equation*}

Let $M^{4n}$ be a quoric manifold with characteristic pair $(P^n,\lambda)$. Let $\cF(P^n)$ denote the poset of nonempty faces of $P^n$, ordered by inclusion:
\begin{equation*}
F\preceq G \quad\Longleftrightarrow\quad
F\subseteq G.
\end{equation*}
We equip $\cF(P^n)$ with the corresponding upper Alexandrov topology and denote the resulting finite space by $\cF_{\preceq}$. Its minimal open neighborhoods are
\begin{equation*}
U_F=\{G\in\cF(P^n)\mid F\subseteq G\}.
\end{equation*}
\noindent
Although the topology of $\cF_{\preceq}$ depends only on $P^n$, the characteristic functor $\lambda$ equips this finite space with additional isotropy data. 

The following example illustrates this topology.

\begin{example}\label{ex:HP2}
Let $M=\mathbb{HP}^2$ with orbit space $P^2=\Delta^2$ the standard $2$-simplex, with vertices $v_0,v_1,v_2$ and edges $e_{01},e_{02},e_{12}$. The face poset $\cF(\Delta^2)$ consists of seven elements: the three vertices, the three edges, and the top face $F=\Delta^2$ itself, ordered by inclusion, namely $v_i \subset e_{ij} \subset F$ for $i,j\in\{0,1,2\}$ with $i\neq j$. The incidence relations in $\cF(\Delta^2)$ can be visualized  Hasse diagram below.

\begin{center}
\begin{tikzpicture}[
    every node/.style={font=\small},
    vertex/.style={circle, fill, inner sep=1.5pt},
    face/.style={draw, rounded corners, inner sep=4pt}
]
    \node[face] (F) at (0,3) {$\Delta^2$};

    \node[face] (e01) at (-3,1.5) {$e_{01}$};
    \node[face] (e02) at (0,1.5) {$e_{02}$};
    \node[face] (e12) at (3,1.5) {$e_{12}$};

    \node[face] (v0) at (-3,0) {$v_0$};
    \node[face] (v1) at (0,0) {$v_1$};
    \node[face] (v2) at (3,0) {$v_2$};

    \draw (v0) -- (e01);
    \draw (v1) -- (e01);

    \draw (v0) -- (e02);
    \draw (v2) -- (e02);

    \draw (v1) -- (e12);
    \draw (v2) -- (e12);

    \draw (e01) -- (F);
    \draw (e02) -- (F);
    \draw (e12) -- (F);
\end{tikzpicture}
\end{center}

For example, some of the minimal open neighborhoods are

\begin{equation*}
    U_{v_0}=\{v_0,e_{01},e_{02},F\},\qquad U_{e_{01}}=\{e_{01},F\},\qquad U_F=\{F\}.
\end{equation*}
\end{example}

\begin{lemma}\label{lem:Alexandrov-T0}
The Alexandrov space $\cF_{\preceq}$ associated to the face poset $(\cF(P^n),\preceq)$ is a $T_0$ space. In other words, for any two distinct points of $\cF_{\preceq}$, there exists an open neighbourhood of one that does not contain the other.
\end{lemma}

\begin{proof}
The Alexandrov topology on $\cF_{\preceq}$ is induced by the partial order on the face poset of $P^n$. The result therefore follows directly from \cite[Lemma 1.14]{MayFiniteSpaces2008}, which states that the Alexandrov space associated to a partially ordered set is a $T_0$ space.
\end{proof}
\medskip

Given any poset $(\mathcal P,\preceq)$, its \emph{order complex} $\Delta(\mathcal P,\preceq)$ is the abstract simplicial complex whose $k$-simplices are the chains $x_0 \prec x_1 \prec \cdots \prec x_k$ in $\mathcal P$. Applied to the face poset of a polytope, $\Delta(\cF(P^n),\preceq)$ is exactly the \emph{barycentric subdivision} of $P^n$: a chain of faces $F_0\prec\cdots\prec F_k$ corresponds to the simplex spanned by the barycentres of $F_0,\dots,F_k$. In particular, its geometric realisation is  homeomorphic to $P^n$ itself; symbolically
\begin{equation}
P^n \;\cong\; \big|\Delta(\cF(P^n),\preceq)\big|. \label{eq:barycentric}
\end{equation}

In \cite{McCord66}, the geometric realisation of order complex of \emph{any} poset is shown to be weakly homotopically equivalent to the poset itself, equipped with the Alexandrov topology. To be more precise, the map
\begin{equation*}
\mu_{\mathrm{MC}} \colon \big|\Delta(\mathcal P,\preceq)\big| \longrightarrow (\mathcal P,\preceq)
\end{equation*}
sending the interior of the simplex for a chain $x_0\prec\cdots\prec x_k$ to its maximum $x_k$ is a weak homotopy equivalence, see \cite[Theorem~1]{McCord66}. 

\begin{lemma}\label{lem:mccord-open}
The map $\mu_{\mathrm{MC}}\colon |\Delta(\cF(P^n),\preceq)| \to (\cF(P^n),\preceq)$ is open.
\end{lemma}
\begin{proof}
Let $O\subseteq  |\Delta(\cF(P^n),\preceq)|$ be an open set. We will show that $\mu_{\mathrm{MC}}(O)$is an up-set in $(\cF(P^n),\preceq)$.

Suppose $F\in\mu_{\mathrm{MC}}(O)$. Then there exists $q\in O$ lying in the relative interior of a simplex corresponding to a chain
\begin{equation*}
F_0\prec\cdots\prec F_k=F.
\end{equation*}
Let $G\succeq F$. If $G=F$, then $G\in\mu_{\mathrm{MC}}(O)$ already. Otherwise,  appending $G$ gives the chain
\begin{equation*}
F_0\prec\cdots\prec F_k=F\prec G,
\end{equation*}
which determines a simplex having the original simplex as a face.

 In the standard geometric realisation, the original simplex is contained in the closure of the relative interior of this larger simplex. Consequently, every neighborhood of $q$ intersects the relative interior of the larger simplex. Since $O$ is open and $q\in O$, there exists a point $q'\in O$ in that relative interior. By the definition of $\mu_{\mathrm{MC}}$, we have
 \begin{equation*}
    \mu_{\mathrm{MC}}(q')=G.
 \end{equation*}
Hence $G\in\mu_{\mathrm{MC}}(O)$.    We have shown that $F\in\mu_{\mathrm{MC}}(O)$ and $F\preceq G$ imply $G\in\mu_{\mathrm{MC}}(O)$. Thus $\mu_{\mathrm{MC}}(O)$ is an up-set, and therefore open in the upper Alexandrov topology. It follows that $\mu_{\mathrm{MC}}$ is open.
\end{proof}

We regard $\cF(P^n)$ as the finite Alexandrov space associated with the face-inclusion order $\preceq$. The following proposition relates this geometric description to the barycentric subdivision and the McCord map introduced above. 

\begin{proposition}\label{prop:polytope-to-poset}
The map $\rho_{P}\colon P^n \to (\cF(P^n),\preceq)$ defined by the composition
\begin{equation}\label{def:rho-P-to-F}
P^n \;\xrightarrow[\ \cong\ ]{\eqref{eq:barycentric}}\; \big|\Delta(\cF(P^n),\preceq)\big| \;\xrightarrow{\ \mu_{\mathrm{MC}}\ }\; (\cF(P^n),\preceq)
\end{equation} 
is  continuous and open.
\end{proposition}

\begin{proof}
Under the identification $P^n\cong|\Delta(\cF(P^n),\preceq)|$, the map $\rho_P$ agrees with the McCord map $\mu_{\mathrm{MC}}$. Indeed, the relative interior of a simplex in the barycentric subdivision corresponding to a chain $F_0\prec\cdots\prec F_k$ lies in $\mathring F_k$.

The McCord map is a weak equivalence, \cite[Theorem~1]{McCord66}, hence a continuous map by definition, and open by Lemma~\ref{lem:mccord-open}. Therefore $\rho_P$ is continuous and open.
\end{proof}

\subsection{Stratifications}

Here we discuss stratifications defined on a quoric manifold. First we recall the definition of a stratification of a topological space from \cite{Woolf26} which underlines any notion of stratification.

\begin{definition}\label{def: stratification}
A stratification of a topological space $X$ is a set $\{X_i \mid i \in I\}$ of non-empty, disjoint subspaces $X_i$, referred to as the \emph{strata} of the decomposition, such that $\bigcup_{i \in I} X_i = X$, and the following conditions are satisfied:
\begin{itemize}
  \item[(1)] \textbf{Local finiteness:} each $x \in X$ has an open neighborhood $U_x$
  intersecting only finitely many strata, i.e., $\{\, i \in I \mid U_x \cap X_i \neq \varnothing \,\}$ is finite.
  \item[(2)] \textbf{Local closure:} each $X_i$ is locally closed, i.e., an intersection of an open and a closed set.
  
  \item[(3)] \textbf{Frontier condition:} if $X_i \cap \overline{X_j} \neq \varnothing$ then
  $X_i \subset \overline{X_j}$. Equivalently, the closure of each stratum is a union of strata.
\end{itemize}
\end{definition}

Let $M^{4n}$ be a quoric manifold. It is a compact smooth manifold, and the $Q^n$-action is smooth, so the orbits decompose $M^{4n}$ into submanifolds, one for each orbit. Since $Q^n$ is compact, the action is proper, and the stabilisers $Q^n_x$ are closed subgroups, well defined up to conjugacy along each orbit. Grouping points whose stabilisers lie in the same conjugacy class $(H)$ gives the orbit-type decomposition, and taking connected components of its pieces gives the orbit-type stratification. For a smooth proper action this is a Whitney stratification \cite[Section~2.7]{duistermaat1999lie}. Compactness of $M^{4n}$ implies that only finitely many orbit types occur, so the stratification is finite. Each orbit-type stratum has an explicit description as a associated bundle over a fixed-point set \cite[Theorem~1.31]{Meinrenken2003GroupActions}. We summarise this in the following result.

\begin{proposition}\label{prop:orbit-type-stratification}
For a quoric manifold $M^{4n}$ the following hold.
\begin{enumerate}
    \item[(1)] $M^{4n}=\bigsqcup_{(H)}M_{(H)}$ is a finite decomposition, where $M_{(H)}=\{m\in M^{4n}\mid (Q^n_m)=(H)\}$;
    \item[(2)] $M^{4n}$ is Whitney stratified by the connected components of the $M_{(H)}$;
    \item[(3)] $M_{(H)}\cong M_H\times_K Q^n/H $ via $[(m,qH)]\mapsto q\cdot m$, where $M_H:=\{m \in M^{4n}|Q^n_{m}=H\}$ $K:=N_{Q^n}(H)/H$ acts on $Q^n/H$ by $[n]\cdot qH:=qn^{-1}H$.
\end{enumerate}
\end{proposition}

Next, we will define another stratification on the quoric manifold induced from the natural stratification of the orbit space which is identified with a simple $n$-polytope. Recall from Remark~\ref{rem:orbit-space-diff-polytope}, that the orbit space of a quoric manifold is homeomorphic to a simple polytope as topological space with  as a manifold with corners. Combined with the discussion above,  the relevant maps fit into the commutative diagram
\begin{equation}
\begin{tikzcd}[column sep=large]
M^{4n}
\arrow[r,"\pi"]
\arrow[dr,"\pi_{\cF}"']
&
P^n
\arrow[d,"\rho_P"]
\\
&
\cF_{\preceq}
\end{tikzcd}
\end{equation} where $\pi_{\cF}:=\rho_P\circ\pi$. Recent work \cite{Woolf26}, relates posets and stratifications. 

\begin{definition}\cite{Woolf26} A poset-stratified space is a continuous map $\pi\colon X \to I_{\preceq}$ from a topological space $X$ to the Alexandrov space $I_{\preceq}$ of a poset $(I, \preceq)$.
\end{definition}

\begin{theorem}\cite[Theorem B]{Woolf26}\label{thm:poset-stratified-yields-stratified} Suppose $\pi\colon X \to I_{\preceq}$ is a poset-stratified space such that the Alexandrov space $I_{\preceq}$ is locally finite and $\pi$ is open. Then the decomposition $\{\pi^{-1}(i)| i\in \mathrm{Im}(\pi)\}$ is a stratification of $X$.
\end{theorem}
For a quoric manifold $M^{4n}$ and a face $F$ of $P^n$, the isotropy class $(Q^n_x)$ is the same for all $x\in\pi^{-1}(\mathring F)$. Indeed, for $y\in\mathring F$ the fibre $\pi^{-1}(y)$ is a single orbit, so the class is constant on it. By the regular-chart models, $y\mapsto(Q^n_x)$, $x\in\pi^{-1}(y)$, is locally constant on $\mathring F$, and $\mathring F$ is connected. We write $\lambda_M(F)$ for this class. By the local models, $\lambda_M$ is a characteristic functor, and $(P^n,\lambda_M)$ is the \emph{characteristic pair} of $M^{4n}$. Henceforth we write $\lambda=\lambda_M$. Specialising the above to our setup we have the following.

\begin{theorem}\label{thm:quoric-stratified}
Let $M^{4n}$ be a quoric manifold over a simple polytope $P^n$ with characteristic pair $(P^n,\lambda)$. Denote by $\cF_{\preceq}$ the Alexandrov space of the face poset $(\cF(P^n),\preceq)$. Then:
\begin{enumerate}
\item[(1)] the map $\pi_{\cF}:=\rho_P\circ\pi\colon M^{4n}\to\cF_\preceq$, which sends $x$ to $F$ if $\pi(x)\in\mathring F$, is a poset-stratified space;
\item[(2)] the fibres $\{M_F:=\pi_{\cF}^{-1}(F)\mid F\in\mathrm{Im}(\pi_{\cF})\}$ form a finite stratification of $M^{4n}$;
\item[(3)] for every $x\in M^{4n}$, $\lambda(\pi_{\cF}(x))=(Q^n_x)$, and if $\pi_{\cF}(x_1)\subseteq\pi_{\cF}(x_2)$ then $\lambda^b(\pi_{\cF}(x_2))\subseteq\lambda^b(\pi_{\cF}(x_1))$ for the subgroup representatives.
\end{enumerate}
\end{theorem}

\begin{proof}
(1) $\pi$ is continuous and $\rho_P$ is continuous by Proposition~\ref{prop:polytope-to-poset}, so $\pi_{\cF}$ is continuous.

(2) $\pi=\bar\pi\circ p$ with $\bar\pi$ a homeomorphism and $p$ the quotient map of a group action, which is open; so $\pi$ is open. Also $\rho_P$ is open by Proposition~\ref{prop:polytope-to-poset}. Hence $\pi_{\cF}$ is open. Since $P^n$ has finitely many faces, $\cF(P^n)$ is finite, hence locally finite. By Theorem~\ref{thm:poset-stratified-yields-stratified}, $\{\pi_{\cF}^{-1}(F)\}_{F}$ is a stratification of $M^{4n}$.

(3) For $x\in M_F$ we have $\pi(x)\in\mathring F$, so $(Q^n_x)=\lambda(F)$ by the definition of $\lambda$. The compatibility of the representatives with reverse face inclusion follows from the canonical construction of $(P^n,\lambda)$.
\end{proof}

\begin{remark}
The Alexandrov topology on $\cF_{\preceq}$ depends only on the face-inclusion relations of $P^n$, whereas $\lambda$ encodes the isotropy data of the $Q^n$-action. The notion of a poset-stratified space does not require $\lambda$; the compatibility in Theorem~\ref{thm:quoric-stratified}(3) is a feature of the quoric structure.
\end{remark}

\begin{lemma}\label{lem:isotropy-of-strata}
Let $\pi\colon M^{4n}\to P^n$ be a smooth locally regular quoric manifold over a simple polytope. For each nonempty face $F\subseteq P^n$ let $M_F=\pi^{-1}(\mathring F)$, and let $\lambda^b(F)\le Q^n$ be a chosen representative of $\lambda(F)$. Then:
\begin{enumerate}
\item[(1)] $M_F$ is a $Q^n$-invariant smooth embedded submanifold of $M^{4n}$;
\item[(2)] every point of $M_F$ has stabiliser class $\lambda(F)$, and if $\mathrm{codim}F=k$, each such stabiliser is isomorphic to $Q^k$;
\item[(3)] there is a noncanonical $Q^n$-equivariant diffeomorphism
\begin{equation}\label{eq:orbit-type-slice}
M_F\cong\mathring F\times Q^n/\lambda^b(F),
\end{equation}
where $Q^n$ acts trivially on $\mathring F$ and by left translation on $Q^n/\lambda^b(F)$.
\end{enumerate}
\end{lemma}

\begin{proof}
(1) $M_F$ is $Q^n$-invariant since $\pi$ is the orbit map. By Definition~\ref{def:loc_reg}, in a regular chart $M_F$ corresponds to the preimage of the relative interior of a face of a regular corner, which is an embedded submanifold.

(2) The first claim is the definition of $\lambda$. For the second, choose a vertex $v\in F$. The restriction of $\lambda$ to the faces containing $v$ is an acceptable isotropy functor, so Remark~\ref{rem:acceptable-isotropy} gives that $\lambda(F)$ has rank $k=\mathrm{codim}F$, that is, each stabiliser is isomorphic to $Q^k$.

(3) Let $G=Q^n$, $H=\lambda^b(F)$, $K=N_G(H)/H$ and $E_F=\{m\in M_F: G_m=H\}$. By Proposition~\ref{prop:orbit-type-stratification}(3), $E_F\to M_F/G$ is a smooth principal $K$-bundle and $M_F\cong E_F\times_K G/H$. The regular-chart models identify $M_F/G$ diffeomorphically with $\mathring F$, which is contractible and paracompact, so $E_F\cong\mathring F\times K$ as principal bundles. Hence $M_F\cong\mathring F\times G/H$, and this diffeomorphism is $G$-equivariant because $G$ acts on the factor $G/H$ by left translation while $K$ acts on the right.
\end{proof}

The preceding results show that the face stratification carries two complementary structures. The finite space $\cF_{\preceq}$ records the incidence relations among the strata, while the characteristic functor $\lambda$ records their isotropy classes and the corresponding compatibility relations. Thus the natural combinatorial structure associated with the stratification is the pair $(\cF_{\preceq},\lambda)$.

\begin{remark}[Relation with the canonical model]
\label{rem:face-strata-canonical-model}
The decomposition \eqref{eq:orbit-type-slice} may be viewed as the restriction of Hopkinson's canonical construction to the relative interior of a face. Indeed, let $(P^n,\lambda)$ be the characteristic pair of $M^{4n}$ and $\mathcal M(\lambda)=(Q^n\times P^n)/{\sim_\lambda}$ the canonical model . Since the isotropy subgroup $\lambda^b(F)$ is constant over $\mathring F$, the restriction of the canonical model to this relative interior admits a natural $Q^n$-equivariant identification $\pi_\lambda^{-1}(\mathring F)\cong\mathring F\times Q^n/\lambda^b(F)$.
\end{remark}

We now compare this stratification with the classical orbit-type stratification of the $Q^n$-action.

\begin{corollary}[Comparison of the face and orbit-type stratifications]
\label{cor:face-orbit-type-comparison} Let $M^{4n}$ be a smooth quoric manifold with a smooth locally regular $Q^n$-action and characteristic pair $(P^n,\lambda)$. Then, for every isotropy conjugacy class $(H)$, holds
\begin{equation*}
M_{(H)}=\bigsqcup_{\substack{F\in\cF(P^n)\\
\lambda(F)=(H)}}M_F
\end{equation*} The sets $M_F$ occurring in this decomposition are precisely the connected components of $M_{(H)}$.
\end{corollary}

Hence the face stratification agrees with the connected orbit-type stratification, although the partition indexed only by isotropy conjugacy classes may have fewer, disconnected pieces. We close this section describing the above decompositions to the case of $\HP{2}$.

\begin{example}\label{ex:HP2-stratifications}
Consider $\mathbb{HP}^2$ with the action of $Q^2$ and orbit space homeomorphic to $P^2=\Delta^2$. Denote the vertices by $v_0,v_1,v_2$, the edges by $e_{01},e_{02},e_{12}$ and the $\Delta^2$ being the 2-face. The isotropy subgroups $\lambda(F)$ are:
\begin{equation*}
\lambda(v_i)\cong Q^2 \ (\text{fixed points}), \qquad \lambda(e_{ij})\cong Q \ (\text{rank-one, one per edge}), \qquad \lambda(\Delta^2)=\{e\}.
\end{equation*}
\begin{itemize}
\item(Face partition) Pulling up the stratification of $\Delta^2$ by relative interiors of faces gives the stratification
\begin{equation*}
\mathbb{HP}^2 = M_{\Delta^2} \;\sqcup\; \big(M_{e_{01}} \sqcup M_{e_{02}} \sqcup M_{e_{12}}\big) \;\sqcup\; \big(M_{v_0}\sqcup M_{v_1}\sqcup M_{v_2}\big),
\end{equation*} which consists of seven strata in total, one for each face of $\Delta^2$. In particular, $M_{v_i}=\{[\text{coordinate point}]\}$ is a single fixed point, each $M_{e_{ij}}\cong S^3\times(0,1)$, and $M_{\Delta^2}$ is the open dense stratum of free orbits.
\item(Orbit-type partition) The isotropy subgroups $\lambda(e_{01}),\lambda(e_{02}),\lambda(e_{12})\le Q^2=Q\times Q$ are pairwise non-conjugate and occur as (conjugates of) $1\times Q$, $Q\times 1$, and the diagonal $\Delta Q=\{(q,q)\}$ respectively. By Proposition~\ref{prop:orbit-type-stratification}, grouping the face strata by conjugacy class of isotropy gives the orbit-type partition
\begin{equation*}
\mathbb{HP}^2 = M_{(e)} \;\sqcup\; M_{(1\times Q)} \;\sqcup\; M_{(Q\times 1)} \;\sqcup\; M_{(\Delta Q)} \;\sqcup\; M_{(Q^2)},
\end{equation*}
which consists of five subsets, with
\begin{equation*}
    M_{(e)} = M_{\Delta^2}, \qquad M_{(\lambda(e_{ij}))} = M_{e_{ij}},
\end{equation*}
\begin{equation*}
    M_{(Q^2)} = M_{v_0}\sqcup M_{v_1}\sqcup M_{v_2} = \{[1:0:0],[0:1:0],[0:0:1]\}.
\end{equation*}
\end{itemize}
\end{example}

\subsection{The Borel space of a quoric manifold}
\label{sec:borel-construction} Here we define the Borel construction of a quoric manifold $M^{4n}$. It is a stratified space, with the stratification induced from the face stratification of $M^{4n}$, and each of its strata is connected. We also prove a Borel analogue of the equivariant deformation retraction of Lemma~\ref{lemma:deformation-retract}.
First, we fix some notation. Let
\begin{equation*}
    E_NQ^n:=(S^{4N+3})^n\subset E_{N+1}Q^n,
\end{equation*}
so that $EQ^n=\bigcup_NE_NQ^n$ carries the colimit topology, and put
\begin{equation*}
M_Q^N:=E_NQ^n\times_{Q^n}M^{4n},
\end{equation*}
which is a compact manifold. The \emph{Borel construction} (or \emph{homotopy quotient}) of $M^{4n}$ is
\begin{equation*}
M_Q:=EQ^n\times_{Q^n}M^{4n}=\bigcup_NM_Q^N,
\end{equation*}
where the inclusions $M_Q^N\subset M_Q^{N+1}$ are closed embeddings of compact smooth manifolds, and $M_Q$ carries the colimit topology. Since $M^{4n}$ is compact Hausdorff, the product $EQ^n\times M^{4n}$ carries the colimit topology of the spaces $E_NQ^n\times M^{4n}$. Hence the colimit topology on $M_Q$ coincides with the quotient topology induced by $q\colon EQ^n\times M^{4n}\to M_Q$. In particular, $q$ is an open surjection. The space $M_Q$ is connected, being the image of the connected space $EQ^n\times M^{4n}$ and paracompact \cite[12.2.1]{Bernstein1994}

The projection to the Alexandrov space of the face poset is
\begin{equation*}
b_{Q}\colon M_Q\to\cF_{\preceq},\qquad [e,m]\mapsto \pi_{\cF}(m),
\end{equation*}
which is continuous because $\pi_{\cF}$ is. For a face $F$ put
\begin{equation*}
V_F:=\pi_{\cF}^{-1}(U_F)=\pi^{-1}\Big(\bigcup_{G\supseteq F}\mathring G\Big)=\bigsqcup_{G\supseteq F}M_G,\qquad V_{F,Q}:=b_Q^{-1}(U_F)=EQ^n\times_{Q^n}V_F.
\end{equation*}
Here $U_F\subseteq\cF_\preceq$ is the set of faces containing $F$, and $\bigcup_{G\supseteq F}\mathring G\subseteq P^n$ is the corresponding subset of $P^n$ (its image under $\rho_P^{-1}$); we denote both by the same symbol. Both $V_F\subseteq M^{4n}$ and $V_{F,Q}\subseteq M_Q$ are open, since $U_F$ is open in $\cF_\preceq$.

Accordingly, $V_{F,Q}$ is the disjoint union of the locally closed pieces $EQ^n\times_{Q^n}M_G$, $G\supseteq F$. By \eqref{eq:orbit-type-slice} applied to each stratum,
\begin{equation*}
EQ^n\times_{Q^n}M_G\;\cong\;\mathring G\times\big(EQ^n/\lambda^b(G)\big),
\end{equation*}
and $EQ^n/\lambda^b(G)$ is a model for $B\lambda^b(G)$, because $\lambda^b(G)$ acts freely on $EQ^n$ and $Q^n$ acts trivially on $\mathring G$. We emphasise that this is a decomposition into pieces, not a product decomposition of $V_{F,Q}$.
\begin{proposition}\label{prop:loc-on-orbit-type-strata}
For each face $F$ of $P^n$, the space $EQ^n\times_{Q^n}M_F$ is connected and homotopy equivalent to $B\lambda^b(F)$. In particular $\pi_1\big(EQ^n\times_{Q^n}M_F\big)=0$.
\end{proposition}
\begin{proof}
Put $H=\lambda^b(F)\cong Q^k$, $k=\mathrm{codim}F$. By \eqref{eq:orbit-type-slice}, and since $Q^n$ acts trivially on $\mathring F$ and $EQ^n\times_{Q^n}Q^n/H\cong EQ^n/H$,
\begin{equation*}
EQ^n\times_{Q^n}M_F\;\cong\;\mathring F\times EQ^n/H\;\simeq\;BH,
\end{equation*}
the last step because $\mathring F$ is contractible. The space $BH$ is connected, and $\pi_1(BH)\cong\pi_0(H)=1$ by Lemma~\ref{lem:pi1-vanishing}, as $H$ is connected.
\end{proof}

\begin{proposition}[Borel deformation retraction]
\label{prop:borel-deformation-retract}
Let $F$ be a face of $P^n$, $f\in\mathring F$ and $O_f=\pi^{-1}(f)$. The homotopy $H\colon[0,1]\times V_F\to V_F$ of Lemma~\ref{lemma:deformation-retract} induces a continuous homotopy
\begin{equation*}
H_Q\colon[0,1]\times V_{F,Q}\to V_{F,Q},\qquad H_Q\big(t,[e,y]\big):=\big[e,H(t,y)\big],
\end{equation*}
such that
\begin{enumerate}
\item[(1)] $H_Q(1,\cdot)=\mathrm{id}$ and $H_Q(0,V_{F,Q})=EQ^n\times_{Q^n}O_f\cong EQ^n/\lambda^b(F)$;
\item[(2)] $H_Q(t,\cdot)$ fixes $EQ^n\times_{Q^n}O_f$ pointwise for all $t$;
\item[(3)] for every face $G\supseteq F$ and every $t$, $H_Q(t,\cdot)$ preserves the closure of $EQ^n\times_{Q^n}M_G$ in $V_{F,Q}$.
\end{enumerate}

In particular $V_{F,Q}$ deformation retracts onto $EQ^n\times_{Q^n}O_f$.
\end{proposition}

\begin{proof}
\emph{Well-definedness.} Suppose $[e,y]=[e',y']$, so $e'=eg^{-1}$ and $y'=g\cdot y$ for some $g\in Q^n$. Since $H(t,\cdot)$ is equivariant,
\begin{equation*}
[e',H(t,y')]=[eg^{-1},g\cdot H(t,y)]=[e,H(t,y)].
\end{equation*}

\emph{Continuity.} Let $q_F:=q|_{EQ^n\times V_F}\colon EQ^n\times V_F\to V_{F,Q}$. It is an open surjection, being the restriction of the open map $q$ to the saturated open set $EQ^n\times V_F$. Hence $\mathrm{id}_{[0,1]}\times q_F$ is an open surjection, thus a quotient map. The map $(t,e,y)\mapsto q_F(e,H(t,y))$ is continuous and constant on its fibres, so it descends to a continuous $H_Q$.

\emph{(1) and (2).} $H(1,\cdot)=\mathrm{id}$ gives $H_Q(1,\cdot)=\mathrm{id}$. Since $H(0,V_F)=O_f$, we get $H_Q(0,V_{F,Q})=EQ^n\times_{Q^n}O_f$, and $O_f\cong Q^n/\lambda^b(F)$ gives $EQ^n\times_{Q^n}O_f\cong EQ^n/\lambda^b(F)$. As $H(t,\cdot)|_{O_f}=\mathrm{id}$, the same holds for $H_Q$.

\emph{(3).} Let $G\supseteq F$. By the proof of Lemma~\ref{lemma:deformation-retract}, $\overline{M_G}\cap V_F=\pi^{-1}(G)\cap V_F=:A_G$ is $Q^n$-invariant and $H_t(A_G)\subseteq A_G$, so $H_{Q,t}$ preserves $EQ^n\times_{Q^n}A_G$. The latter is the closure of $EQ^n\times_{Q^n}M_G$ in $V_{F,Q}$, because $q_F$ is open and $EQ^n\times M_G$ is saturated, so $q_F(\overline S)=\overline{q_F(S)}$ for saturated $S$.
\end{proof}

\begin{corollary} 
$V_{F,Q}\simeq B\lambda^b(F)$ and hence $V_{F,Q}$ is connected and $\pi_1(V_{F,Q})=0$.
\end{corollary}
\begin{proof}
$H_Q$ is a  deformation retraction of $V_{F,Q}$ onto $EQ^n\times_{Q^n}O_f\cong EQ^n/\lambda^b(F)\simeq B\lambda^b(F)$. The latter is connected and simply connected by the proof of Proposition~\ref{prop:loc-on-orbit-type-strata}
\end{proof}

\section{Derived categories of sheaves}
\label{sect:sheaves-quoric}

We study quoric manifolds using sheaves that retain information about the $Q^n$-action and the face stratification of Theorem~\ref{thm:quoric-stratified}. We first recall the necessary background on sheaves, equivariant sheaves, and local systems, and then specialise to quoric manifolds, describing the corresponding categories in several cases. For further background, see \cite{kashiwara2013sheaves}, \cite{dimca2004sheaves}, and \cite{Achar2021PerverseSheaves}.

\subsection{On sheaves}\label{sec:on-sheaves}

We recall the basic notions of sheaf theory needed in the sequel, including stalks, supports, sheaf cohomology, and acyclic resolutions. We also introduce sheaves of DG-algebras and DG-modules, which will be used in our description of equivariant derived categories.

\begin{definition}\label{def:sheaf}
Let $X$ be a topological space.  A \emph{sheaf of $\C$-modules} on $X$ is an assignment $F$ that associates
\begin{itemize}
\item[(1)] to each open set $U\subset X$ a $\C$-module $F(U)$;
\item[(2)] to each pair of open sets $U\subset V$ a homomorphism of $\C$-modules denoted by $\mathrm{res}^{F}_{V,U}\colon F(V) \to F(U)$, called the \emph{restriction map},
\end{itemize} 
such that $\mathrm{res}^{F}_{U,U}=\mathrm{id}_{F(U)}$ and $\mathrm{res}^{F}_{W,U}=\mathrm{res}^{F}_{V,U}\circ \mathrm{res}^{F}_{W,V}$ for all $U\subset V \subset W$ and moreover  if for every open $U\subset X$ and every open cover $U=\bigcup_{i\in I}U_i$ then the following hold:
\begin{description}
\item[(Locality)] if $s\in F(U)$ satisfies $s|_{U_i}=0$ for every $i\in I$, then $s=0$;
\item[(Gluing)] given $s_i\in F(U_i)$ for $i\in I$ with $s_i|_{U_i\cap U_j}=s_j|_{U_i\cap U_j}$ for all $i,j\in I$, there is $s\in F(U)$ with $s|_{U_i}=s_i$ for every $i\in I$.
\end{description}

\end{definition}

 An element $s\in F(U)$ is a \emph{section} of $F$ over $U$.

\begin{notation}
 For a topological space $X$, we write $\mathrm{Sh}(X)$ for the category of sheaves on $X$.
\end{notation}

\begin{definition}\label{def:stalk}
The \emph{stalk} of a sheaf $F$ at $x$ is the $\C$-module $        F_x:=\mathop{\varinjlim}\limits_{\substack{x\in U}} F(U)$. An element of $F_x$ is a \emph{germ} of $F$ at $x$.
\end{definition}

\begin{remark}
For a finite space $X$, the direct limit defining $F_x$ is attained at the minimal open neighbourhood $U_x$ of $x$, so $F_x\cong F(U_x)$.
\end{remark}

\begin{definition}\label{def:support}
Let $F$ be a sheaf on $X$. The \emph{support} of a section $s\in F(U)$ is
\begin{equation*}
\mathrm{supp}(s)=\overline{\{x\in U:s_x\neq 0\}}.
\end{equation*}
The \emph{support} of $F$ is
\begin{equation*}
\mathrm{supp}(F)=\overline{\{x\in X:F_x\neq 0\}}.
\end{equation*}
\end{definition}

The fundamental example of a sheaf for this paper is the constant sheaf which we introduce below.

\begin{example}[Constant sheaf]\label{ex:constant-sheaf}
The \emph{constant sheaf} $\underline{A}_X$ associated to a $\C$-module $A$ is defined by
\begin{equation*}
\underline{A}_X(U):=\{\text{locally constant functions } U\to A\},
\end{equation*}
where $A$ carries the discrete topology, with restriction of functions as restriction maps. Its stalk at every $x\in X$ is naturally isomorphic to $A$. If $U$ is connected and nonempty, then $\underline{A}_X(U)\cong A$. More generally, when the connected components of $U$ are open, one has $\underline{A}_X(U)\cong A^{\pi_0(U)}$.
\end{example}

The sheaf-cohomology computations later in the paper rest on a small further package of standard facts, which we record next: softness, resolutions, and the abstract de Rham theorem.

\begin{definition}\label{def:soft-sheaf}
Let $F$ be a sheaf of $\C$-modules on $X$. For a closed subset $Z\subset X$ write $F(Z):=\varinjlim_{U\supset Z}F(U)$, the direct limit over open neighbourhoods of $Z$. The sheaf $F$ is \emph{soft} if the restriction map $F(X)\to F(Z)$ is surjective for every closed $Z\subset X$, i.e.\ every section of $F$ over a closed subset extends to a global section.
\end{definition}

\begin{definition}\label{def:sheaf-cohomology}
The \emph{sheaf cohomology} of a sheaf $F$ on $X$ is $H^i(X,F):=R^i\Gamma(X,-)(F)$, the $i$-th right derived functor of the global sections functor $\Gamma(X,-)\colon\mathrm{Sh}(X)\to\C\text{-mod}$, computed via an injective resolution of $F$ in the abelian category $\mathrm{Sh}(X)$. We say $F$ is \emph{acyclic} if $H^i(X,F)=0$ for all $i>0$.
\end{definition}

\begin{definition}\label{def:resolution}
A \emph{resolution} of a sheaf $F$ on $X$ is a complex of sheaves $G^\bullet=(G^0\xrightarrow{d^0}G^1\xrightarrow{d^1}\cdots)$ together with a morphism $\varepsilon\colon F\to G^0$ such that
\begin{equation*}
0\to F\xrightarrow{\varepsilon}G^0\xrightarrow{d^0}G^1\xrightarrow{d^1}G^2\to\cdots
\end{equation*}
is exact on stalks at every point of $X$. The resolution is \emph{acyclic} if every $G^j$ is acyclic.
\end{definition}

\begin{proposition}\label{prop:soft-acyclic}
Let $X$ be paracompact.
\begin{enumerate}
\item[(1)] Every soft sheaf on $X$ is acyclic \cite[II.9.11]{bredon_sheaf_1997}.
\item[(2)] If $\mathcal O$ is a soft sheaf of $\C$-algebras on $X$ and $M$ is a sheaf of $\mathcal O$-modules, then $M$ is soft \cite[II.9.16, II.9.18]{bredon_sheaf_1997}; in particular $M$ is acyclic by (1).
\end{enumerate}
\end{proposition}

\begin{theorem}[Abstract de Rham theorem, {\cite[II.9]{bredon_sheaf_1997}}]\label{thm:abstract-de-rham}
If $0\to F\to G^\bullet$ is an acyclic resolution of $F$, then $H^i(X,F)\cong H^i\bigl(\Gamma(X,G^\bullet)\bigr)$ for all $i\geq0$, where the right-hand side is the cohomology of the complex of global sections.
\end{theorem}

\begin{example}[Resolution of the constant sheaf]\label{ex:resolution-constant-sheaf}
Let $X$ be a smooth manifold and $\Omega^\bullet_X$ the sheaf of smooth $\C$-valued differential forms, with $d$ the exterior derivative. The Poincar\'e lemma says the sequence
\begin{equation*}
0\to\C_X\to\Omega^0_X\xrightarrow{d}\Omega^1_X\xrightarrow{d}\Omega^2_X\to\cdots
\end{equation*}
is exact on stalks, i.e.\ $\Omega^\bullet_X$ is a resolution of the constant sheaf $\C_X$ of Example~\ref{ex:constant-sheaf}. Each $\Omega^k_X$ is a module over the sheaf of smooth functions $\Omega^0_X$, which is soft (by partitions of unity); by Proposition~\ref{prop:soft-acyclic} the resolution is acyclic, and Theorem~\ref{thm:abstract-de-rham} recovers the classical de Rham theorem $H^i(X,\C)\cong H^i_{dR}(X;\C)$. An analogous construction for the Borel space $M_Q$ will be considered in Lemma~\ref{lem:de-rham-properties}, using the differential-form model specified there.
\end{example}

In Section~\ref{sec:main-section}, we will work with an algebraic refinement of a sheaf valued in differential graded algebras. We conclude this section by recalling the relevant construction from \cite[Section 1]{Lunts1995}.

\begin{definition}\label{def:dga-sheaf}
Let $X$ be a topological space. A \emph{sheaf of differential graded algebras} (sheaf of DG-algebras) on $X$ is defined exactly as in Definition~\ref{def:sheaf}, wwith $\C$-modules replaced throughout by $\mathbb Z$-graded $\C$-algebras: it is a sheaf $\cA=\bigoplus_{n\in\mathbb Z}\cA^n$ of $\mathbb Z$-graded $\C$ algebras on $X$, equipped with a differential $d\colon\cA^n\to\cA^{n+1}$ which is a morphism of sheaves of $\C$-vector spaces, such that $d^2=0$ and for any open $U\subseteq X$ and homogeneous $a,b\in\cA(U)$,
\begin{equation*}
d(ab)=d(a)b+(-1)^{\deg(a)}a\,d(b).
\end{equation*}
\end{definition}

\begin{definition}\label{def:dg-module-sheaf}
Let $\cA$ be a sheaf of DG-algebras on $X$. A \emph{sheaf of differential graded left $\cA$-modules} is a sheaf $\mathcal M=\bigoplus_{n\in\mathbb Z}\mathcal M^n$ of $\mathbb Z$-graded
$\C$-vector spaces on $X$, together with
\begin{enumerate}
\item[(1)] a sheaf morphism giving a graded $\cA$-module structure $\cA\otimes\mathcal M\to\mathcal M$,
\item[(2)] a differential $d_{\mathcal M}\colon\mathcal M^n\to\mathcal M^{n+1}$,
\end{enumerate}
such that $d_{\mathcal M}^2=0$ and, for any open $U\subseteq X$, homogeneous $a\in\cA(U)$ and $m\in\mathcal M(U)$,
\begin{equation*}
d_{\mathcal M}(a\cdot m)=d(a)\cdot m+(-1)^{\deg(a)}a\cdot d_{\mathcal M}(m).
\end{equation*}
\end{definition}

\begin{remark}
The Leibniz rule ensures that the differential interacts correctly with the module action.
\end{remark}

\subsection{Equivariant sheaves and local systems}
\label{sec:equivariant-local systems-sheaves-background}

Here, we  collect the necessary background on equivariant sheaves and local systems that we will need in later sections following \cite{Achar2021PerverseSheaves} and \cite{dimca2004sheaves}. 

Let $G$ be a topological group with multiplication $m\colon G\times G\to G$, acting continuously on a topological space $X$ by 
\begin{equation*}
\alpha \colon G\times X \rightarrow X, \qquad \alpha(g,x)=g\cdot x,
\end{equation*}
Let  $\mathrm{pr}_2 \colon G\times X \rightarrow X$, $\mathrm{pr}_2(g,x)=x$ be the projection map and consider the maps :
\begin{equation}\label{def: maps for equi}
\left.
\begin{aligned}
m\times \mathrm{id} &\colon G\times G\times X \to G\times X, & (g,h,x) &\longmapsto (gh,\,x), \\
\mathrm{id} \times \alpha & \colon G\times G\times X \to G\times X, & (g,h,x) &\longmapsto (g,\,h\cdot x), \\
\mathrm{pr}_{23} &\colon G\times G\times X \to G\times X, \qquad  & (g,h,x) &\longmapsto (h,\,x).
\end{aligned}
\right\}
\end{equation}

\begin{definition}\label{def:equivariant-sheaf}
A \textit{$G$-equivariant sheaf} on $X$ is a pair $(\cF,\theta)$ with $\cF\in\mathrm{Sh}(X)$ and $\theta$ an isomorphism
\begin{equation*}
\theta \colon \mathrm{pr}_2^*\cF \xrightarrow{\ \sim\ } \alpha^*\cF
\end{equation*}
in $\mathrm{Sh}(G\times X)$, satisfying the cocycle condition
\begin{equation}\label{eq:cocycle}
    (\mathrm{id} \times \alpha)^*\theta \circ \mathrm{pr}_{23}^*\theta \;=\; (m\times \mathrm{id})^*\theta
\end{equation}
on $G\times G\times X$.
\end{definition}

On stalks, the equivariant structure induces isomorphisms $\theta_{(g,x)}:\cF_x\longrightarrow \cF_{gx}$. The cocycle condition expresses the commutativity of the following diagram:

\begin{equation*}
\begin{tikzcd}[column sep=large,row sep=large]
\cF_x
  \arrow[r,"\theta_{(gh,x)}"]
  \arrow[d,"\theta_{(h,x)}"']
&
\cF_{ghx}
\\
\cF_{hx}
  \arrow[ur,"\theta_{(g,hx)}"']
&
\end{tikzcd}
\end{equation*}

A map $(\cF,\theta)\to(\cF',\theta')$ is a map of sheaves
$f\colon\cF\to\cF'$ with
$\alpha^*f\circ\theta=\theta'\circ\mathrm{pr}_2^*f$ and the resulting category
$\mathrm{Sh}_G(X)$ is abelian \cite[Proposition~6.2.2]{Achar2021PerverseSheaves}.

Describing the category of equivariant sheaves is not an easy task. The following lemma is the analogue, for topological groups, of \cite[Lemma~7.13]{gourevitch_representation_2021}, combined with \cite[(0.2) Example~2]{Bernstein1994}.

\begin{lemma}\label{lemma: Q-sheaves on orbits}
Let $G$ be a compact Lie group acting transitively on a Hausdorff space $X$, and fix $x\in X$ with stabiliser $G_x$, so that $X\cong G/G_x$. Then
\begin{equation}\label{eq:equiv-sheaves}
\mathrm{Sh}_G(X) \;\simeq\; \mathrm{Sh}_{G_x}(\mathrm{pt}) \;\simeq\; \mathrm{Rep}\big(\pi_0(G_x)\big),
\end{equation}
where an equivariant structure on a point is a locally constant homomorphism $G_x\to\mathrm{Aut}(V)$, hence factors through the finite group $\pi_0(G_x)$. In particular $\mathrm{Sh}_G(G/H)\simeq \C\text{-Mod}$ whenever $H$ is connected.
\end{lemma}

In the case of  $X=Q^n \cdot m$ an orbit of a point $m$ of a quoric manifold $M^{4n}$ the above lemma applies as follows.

\begin{corollary}\label{prop:sheaves-on-orbits}
Let $M^{4n}$ be a quoric manifold and $m\in M^{4n}$. Then $\mathrm{Sh}_{Q^n}(Q^n\cdot m)\simeq\mathrm{Vect}_{\C}$; that is, every $Q^n$-equivariant sheaf on an orbit is constant.
\end{corollary}

\begin{proof}
The orbit is $Q^n\cdot m\cong Q^n/Q^n_m$, where $Q^n_m$ is isomorphic to $(S^3)^k$ for some $0\le k\le n$, by the local structure of quoric manifolds. In particular $Q^n_m$ is connected, so $\pi_0(Q^n_m)=\{e\}$, and Lemma~\ref{lemma: Q-sheaves on orbits} gives $\mathrm{Sh}_{Q^n}(Q^n\cdot m)\simeq\mathrm{Rep}(\{e\})\simeq\mathrm{Vect}_{\C}$.
\end{proof}

When the $G$-action on $X$ is free, equivariant sheaves are simply sheaves on the quotient.

\begin{lemma}[{\cite[Section 0.3]{Bernstein1994}}] \label{lemma: free-descent}
If $G$ acts freely and properly on $X$ with quotient map $\pi\colon X\to X/G$, then $\pi^*\colon\mathrm{Sh}(X/G)\to\mathrm{Sh}_G(X)$ is an equivalence.
\end{lemma}

\begin{remark}
For any $G$-space $Y$ and the classifying bundle $EG\to BG$, the diagonal action on $EG\times Y$ is free, so Lemma~\ref{lemma: free-descent} gives $\mathrm{Sh}_G(EG\times Y)\simeq\mathrm{Sh}(EG\times_G Y)$.
\end{remark}

Next, we introduce the notation of a local system on a topological space.

\begin{definition}
Let $X$ be a topological space. A sheaf $\mathcal L$ on $X$ is \textbf{locally constant}, or a \textbf{local system}, if $X$ admits an open cover $(U_i)_{i\in I}$ such that $\mathcal L|_{U_i}$ is a constant sheaf for every $i$. We write $\mathrm{Loc}(X)\subset\mathrm{Sh}(X)$ for the full subcategory of local systems.
\end{definition}

The following classical equivalence is the basic computational tool for local systems; see \cite[2.5.1]{dimca2004sheaves}.

\begin{proposition}\label{prop:loc-classification}
Let $X$ be path-connected and locally contractible  and $x_0\in X$.  The following categories are equivalent:
\begin{enumerate}
    \item[(1)] local systems of $\C$-modules on $X$;
    \item[(2)] covariant functors from the fundamental groupoid of $X$ to $\C\text{-mod}$;
    \item[(3)] $\C$-linear representations $\rho\colon\pi_1(X,x_0)\to\mathrm{Aut}_\C(M)$ on an $\C$-module $M$.
\end{enumerate}

\end{proposition}
Writing $\mathrm{Rep}(\pi_1(X,x_0))$ for the category in (3), the above proposition is summarised to the following equivalence
\begin{equation}\label{eq:loc-sheaves}
\mathrm{Loc}(X)\;\simeq\;\mathrm{Rep}\big(\pi_1(X,x_0)\big).
\end{equation}

\begin{definition}
    An \textbf{equivariant local system} on a $G$-space $X$ is a $G$-equivariant sheaf whose underlying sheaf is locally constant.
\end{definition}

We write $\mathrm{Loc}_G(X)\subset\mathrm{Sh}_G(X)$ for the corresponding full subcategory. The equivariant analogue of \eqref{eq:loc-sheaves} replaces $X$ by its homotopy quotient.

\begin{proposition}[cf.\ {\cite[Section 2.3]{Joshua2022EquivariantPerverse}}]\label{prop:equivariant-loc}
Let $G$ be a compact Lie group acting on a connected manifold $X$. Denote by $EG\times_G X$ its homotopy quotient and let $[e,x]$ be a basepoint with $x\in X$. Then, the following holds
\begin{equation}\label{eq:equiv-loc-sheaves}
\mathrm{Loc}_G(X)\;\simeq\;\mathrm{Rep}\big(\pi_1(EG\times_G X,\,[e,x])\big).
\end{equation}
\end{proposition}

\begin{remark}\label{rem:local-system-orbits}
For a single orbit $X=G/G_x$ with $G$ connected we have
\begin{equation*}
\mathrm{Loc}_G(G/G_x)\;\simeq\;\mathrm{Rep}\big(\pi_1(BG_x)\big)\;\cong\;\mathrm{Rep}\big(\pi_0(G_x)\big).
\end{equation*}
\end{remark}

For quoric manifolds, the reduction to fundamental-group computations comes down to a single fact: every stabiliser that arises is connected, so the relevant fundamental groups vanish. In particular we have the following lemma about the subgroups of $Q^n$

\begin{lemma}\label{lem:pi1-vanishing}
Let $H\le Q^n$ be a closed subgroup. Then $\pi_1(Q^n/H)\cong\pi_1(BH)\cong\pi_0(H)$. In particular both are trivial when $H$ is connected, for instance when $H$ is conjugate to $Q^k=(SU(2))^k$ for some $0\le k\le n$.
\end{lemma}

\begin{proof}
Since $H$ is closed in the compact Lie group $Q^n$, the quotient maps $Q^n\to Q^n/H$ and $EQ^n\to EQ^n/H$ are principal $H$-bundles; for the second, restricting the free $Q^n$-action on $EQ^n$ to $H$ exhibits $EQ^n$ as a free contractible $H$-space, so $EQ^n/H$ is a model for $BH$. In each case the total space $E\in\{Q^n,EQ^n\}$ has trivial $\pi_0$ and $\pi_1$, since $Q^n\cong SU(2)^n$ is connected and simply connected and $EQ^n$ is contractible. The long exact sequence of the fibration
\begin{equation*}
\cdots\to\pi_1(E)\to\pi_1(E/H)\to\pi_0(H)\to\pi_0(E)\to\cdots
\end{equation*}
therefore gives $\pi_1(E/H)\cong\pi_0(H)$.
\end{proof}

\begin{proposition}\label{prop:loc-on-orbit}
Let $M^{4n}$ be a quoric manifold and $m\in M^{4n}$. Then
\begin{equation*}
\mathrm{Loc}(Q^n\cdot m)\simeq\mathrm{Loc}_{Q^n}(Q^n\cdot m)\simeq\mathrm{Vect}_{\C}.
\end{equation*}
\end{proposition}

\begin{proof}
By Propositions~\ref{prop:loc-classification} and~\ref{prop:equivariant-loc} and Remark~\ref{rem:local-system-orbits}, it suffices to show that the spaces $Y=Q^n\cdot m\cong Q^n/Q^n_m$ and $EQ^n\times_{Q^n}Y\simeq BQ^n_m$ are connected with trivial fundamental group. This is Lemma~\ref{lem:pi1-vanishing}.
\end{proof}

\begin{lemma}\label{lem:strata-simply-connected}
Each face stratum $M_F\cong\mathring F\times Q^n/\lambda^{b}(F)$, see \eqref{eq:orbit-type-slice}, is a connected, simply connected manifold. Hence every local system on $M_F$ is constant, and the only irreducible local system on $M_F$ is $\underline\C_{M_F}$.
\end{lemma}

\begin{proof}
$\mathring F$ is convex, hence contractible. $Q^n/\lambda^{b}(F)$ is connected, and the long exact sequence of $\lambda^{b}(F)\to Q^n\to Q^n/\lambda^{b}(F)$ gives $\pi_1(Q^n/\lambda^{b}(F))=1$, since $\pi_1(Q^n)=\pi_1((S^3)^n)=1$ and $\lambda^{b}(F)$ is connected. Since $M_F$ is a manifold, it is locally contractible, so by Proposition~\ref{prop:loc-classification} $\mathrm{Loc}(M_F)\simeq\mathrm{Rep}(\pi_1(M_F))=\mathrm{Vect}_\C$. Sub-local systems of $\underline V_{M_F}$ correspond to subspaces of $V$, so the irreducible one is $\underline\C_{M_F}$.
\end{proof}

\begin{proposition}\label{prop:loc-on-face-strata}
Let $M^{4n}$ be a quoric manifold with face stratification $\coprod_{F\in\cF(P^n)}M_F$. Then, for every $F\in\cF(P^n)$,
\begin{enumerate}
\item[(1)] $\mathrm{Sh}_{Q^n}(M_F)\simeq\mathrm{Sh}(\mathring F)$;
\item[(2)] $\mathrm{Loc}(M_F)\simeq\mathrm{Loc}_{Q^n}(M_F)\simeq\mathrm{Vect}_{\C}$.
\end{enumerate}
\end{proposition}

\begin{proof}
(1) Put $H=\lambda^b(F)$. By Lemma~\ref{lem:isotropy-of-strata}, $M_F\cong\mathring F\times Q^n/H\cong Q^n\times_H\mathring F$ as $Q^n$-spaces, with $H$ acting trivially on $\mathring F$. Induction along $Q^n\to Q^n/H$, as in Lemma~\ref{lemma: Q-sheaves on orbits}, gives $\mathrm{Sh}_{Q^n}(M_F)\simeq\mathrm{Sh}_H(\mathring F)$. It remains to identify $\mathrm{Sh}_H(\mathring F)$ with $\mathrm{Sh}(\mathring F)$. Since $H$ acts trivially, the action map $\sigma\colon H\times\mathring F\to\mathring F$ equals $\mathrm{pr}_2$, so for $(\mathcal G,\theta)\in\mathrm{Sh}_H(\mathring F)$ the structure map $\theta$ is an automorphism of $\mathrm{pr}_2^*\mathcal G$. Fix $x\in\mathring F$. Restricted to $H\times\{x\}$, $\theta$ is an automorphism of the constant sheaf with fibre $\mathcal G_x$, that is, a locally constant map $H\to\mathrm{Aut}(\mathcal G_x)$. By the cocycle condition its value at $e$ is the identity, and since $H\cong Q^k$ is connected it is identically the identity. As $x$ was arbitrary and a morphism of sheaves is determined by its stalks, $\theta=\mathrm{id}$. Hence the forgetful functor $\mathrm{Sh}_H(\mathring F)\to\mathrm{Sh}(\mathring F)$ is essentially surjective, with section $\mathcal G\mapsto(\mathcal G,\mathrm{id})$. It is fully faithful because for objects with $\theta=\theta'=\mathrm{id}$ the condition $\sigma^*f\circ\theta=\theta'\circ\mathrm{pr}_2^*f$ holds for every sheaf morphism $f$.

(2) Put $H=\lambda^b(F)$. The stratum $M_F$ is a connected manifold, since $\mathring F$ is convex and $Q^n/H$ is a continuous image of the connected group $Q^n$. Also $Q^n$ is a compact Lie group. By Propositions~\ref{prop:loc-classification} and~\ref{prop:equivariant-loc}, it therefore suffices to show that $\pi_1(M_F)$ and $\pi_1(EQ^n\times_{Q^n}M_F)$ vanish, since $\mathrm{Rep}(\{e\})\simeq\mathrm{Vect}_{\C}$.

The first group vanishes by Lemma~\ref{lem:strata-simply-connected}. For the second, by \eqref{eq:orbit-type-slice} and since $Q^n$ acts trivially on $\mathring F$,
\begin{equation*}
EQ^n\times_{Q^n}M_F\;\cong\;\mathring F\times\big(EQ^n\times_{Q^n}Q^n/H\big)\;\cong\;\mathring F\times EQ^n/H\;\simeq\;BH,
\end{equation*}
because $EQ^n/H$ is a model for $BH$ and $\mathring F$ is contractible. Since $H\cong Q^k$ is connected, $\pi_1(BH)\cong\pi_0(H)=1$ by Lemma~\ref{lem:pi1-vanishing}.
\end{proof}

\begin{remark}
Contractibility of $\mathring F$ does not further collapse $\mathrm{Sh}(\mathring F)$; skyscrapers are not constant.
\end{remark}

\subsection{Derived categories, derived functors and t-structures}\label{sec:derived-cats}

We collect the background on triangulated and derived categories, the six-functor formalism, constructible sheaves and recollement, and perverse sheaves that we will need; for derived and triangulated categories in general we refer to \cite{Achar2021PerverseSheaves}. In Section~\ref{sec:equivariant-derived-cat-quoric}, we will specialise to the  equivariant derived category of constructible sheaves on a quoric manifold, and its generators.

\begin{notation} In what follows, we work over the complex numbers $\C$, and all our sheaves are sheaves of $\C$-modules. We denote by $\mathrm{Sh}(X)$ the category of sheaves on a topological space $X$.
\end{notation}

For an abelian category $\mathcal A$, the \textbf{derived category} $D(\mathcal A)$ is obtained from the homotopy category of complexes by formally inverting the quasi-isomorphisms, and $D^{b}(\mathcal A)$ is the full subcategory of complexes with bounded cohomology. Both $D(\mathcal A)$ and $D^{b}(\mathcal A)$ are triangulated categories: the shift is $K^\bullet[n]^i=K^{i+n}$, so that $H^i(K^\bullet[n])\cong H^{i+n}(K^\bullet)$, and every morphism $f^\bullet\colon K^\bullet\to L^\bullet$ sits in a distinguished triangle
\begin{equation}\label{eq:cone-triangle}
    K^\bullet
    \xrightarrow{f^\bullet}
    L^\bullet
    \longrightarrow
    \mathrm{Cone}(f^\bullet)
    \longrightarrow
    K^\bullet[1].
\end{equation}

\begin{notation}
For a topological space $X$ and $\cA:=\mathrm{Sh}(X)$ we write $ D(X):=D(\mathrm{Sh}(X))$ and $D^b(X):=D^b(\mathrm{Sh}(X))$.
\end{notation}

\begin{remark}
A sheaf is an object of $D^b(X)$ by considering it as a complex concentrated in degree zero, so there is a natural embedding $\mathrm{Sh}(X)\hookrightarrow D^b(X)$.
\end{remark}

The derived category is useful because many operations on sheaves are only left or right exact, and passing to complexes allows them to be extended to derived functors. For a continuous map $f\colon X\to Y$, Grothendieck's six-functor formalism yields the derived functors
\begin{equation*}
Rf_*,\ Rf_{!}\colon D^+(X)\longrightarrow D^+(Y),
\qquad
f^{*},\ f^{!}\colon D^+(Y)\longrightarrow D^+(X).
\end{equation*}
These functors send distinguished triangles to distinguished triangles and commute with the shift functor. Moreover, we have the following string of adjunctions

\begin{equation*}
f^*\dashv Rf_*,
\qquad
Rf_!\dashv f^!.
\end{equation*}

\begin{notation}
From now, we will simply write $f_!, f^{*},f_*,  f^{!}$.
\end{notation}

By Corollary~\ref{cor:face-orbit-type-comparison}, the face strata $M_F$ are exactly the connected components of the orbit-type strata, so by Proposition~\ref{prop:orbit-type-stratification}(2) the face stratification $\mathcal S$ is a finite Whitney stratification. Let $Z\subseteq M^{4n}$ be a closed union of strata, $U=M^{4n}\setminus Z$, and $i\colon Z\hookrightarrow M^{4n}$, $j\colon U\hookrightarrow M^{4n}$ the inclusions. This recollement data is most naturally stated already at the level of the bounded-below derived category $D^+(X)$ of sheaves on $X$: $(j_!,j^*,j_*)$ and $(i^*,i_*,i^!)$ are adjoint triples relating $D^+(X)$, $D^+(U)$, and $D^+(Z)$, and every $K\in D^+(X)$ sits in distinguished triangles
\begin{equation*}\label{eq:recollement-triangles}
\begin{aligned}
j_!j^*K &\longrightarrow K
\longrightarrow i_*i^*K
\longrightarrow j_!j^*K[1],\\
i_*i^!K &\longrightarrow K
\longrightarrow j_*j^*K
\longrightarrow i_*i^!K[1]
\end{aligned}
\end{equation*}
see \cite[Exercises 1.3.4, A.7.4]{Achar2021PerverseSheaves}. These functors fit into the recollement diagram
\begin{equation*}
\begin{tikzcd}[column sep=large]
D_c^b(Z)
  \arrow[r,shift left=1.2ex,"i_*"]
&
D_c^b(X)
  \arrow[l,shift left=1.2ex,"i^*"]
  \arrow[l,shift right=1.2ex,"i^!"']
  \arrow[r,shift left=1.2ex,"j^*"]
&
D_c^b(U)
  \arrow[l,shift left=1.2ex,"j_!"]
  \arrow[l,shift right=1.2ex,"j_*"']
\end{tikzcd}
\end{equation*}

For the stratifications considered here these functors preserve constructibility \cite[Lemmas~3.1.4-3.1.8]{Achar2021PerverseSheaves} and the categories $D^b_c(Z)$, $D^b_c(X)$ and $D^b_c(U)$ form a \emph{recollement}; see \cite[Section 1.4]{BBD}. Recollement is what relates constructible sheaves on different strata: the perverse $t$-structure below is built by gluing along it, and $j_!$ and $j_*$ enter the definition of the intermediate extension.

We briefly recall the notion of a $t$-structure on a triangulated category. This will allow us to regard perverse sheaves as the objects of an abelian category inside the constructible derived category.

\begin{definition}
Let $\mathcal{T}$ be a triangulated category. A \textbf{$t$-structure} on $\mathcal{T}$ is a pair of full subcategories $(\mathcal T^{\leq 0},\mathcal T^{\geq 0})$ such that if $\mathcal T^{\leq n}:=\mathcal T^{\leq 0}[-n]$ and $\mathcal T^{\geq n}:=\mathcal T^{\geq 0}[-n]$ then the following hold
\begin{enumerate}
    \item[(1)] for $X\in \cT^{\leq 0}$, $Y\in \cT^{\geq 1}$ $\Hom_{\cT}(X,Y)=0$;
    \item[(2)] $\mathcal T^{\leq 0}\subset\mathcal T^{\leq 1}$ and $\mathcal T^{\geq 0}\supset\mathcal T^{\geq 1}$; 
    \item[(3)] for $X\in\cT$ there is a distinguished triangle $A\to X\to B\to A[1]$ with $A\in\cT^{\leq 0}$ and $B\in\cT^{\geq 1}$.
\end{enumerate}
\end{definition}

The full subcategory $\mathcal H:=\cT^{\leq 0}\cap\cT^{\geq 0}$ is called the \textbf{heart} of the $t$-structure. It is an abelian category. Thus, although $\cT$ is triangulated, its heart carries the usual abelian-category notions such as kernels, cokernels, extensions and simple objects. The standard $t$-structure on the derived category of an abelian category has the original abelian category as its heart.

\subsection{Derived categories of constructible sheaves}

We now recall constructibility with respect to a fixed stratification and describe its consequences for quoric manifolds. Our main observation is that the face strata have trivial fundamental group, so that their local systems are constant. First we introduce the basic ideas, referrring to \cite{Achar2021PerverseSheaves} and \cite{dimca2004sheaves} for more details. In what follows, unless specified otherwise, $X$ will denote a topological space with a stratification $\{X_{i}\}_{i\in I}$ in the sense of Definition~\ref{def: stratification}.

\begin{definition}
Let $X$ be a stratified space in the sense of Definition~\ref{def: stratification}.

\begin{enumerate}
    \item[(1)] A sheaf $\cG\in \mathrm{Sh}(X)$ is said to be constructible with respect the stratification of $X$ if $\cG|_{X_{i}}$ is a local system with stalks to be finitely generated.
    \item[(2)]  A complex $\mathcal G^\bullet\in D^b(X)$ is \textbf{constructible} with respect to the given stratification if each cohomology sheaf $\mathcal H^i(\mathcal G^\bullet)$ is constructible. We denote the resulting full subcategory of $D^b(X)$ by $D^b_c(X)$.
\end{enumerate}
\end{definition}

We will denote by $\mathrm{Sh}_{c}(X)$ the abelian category of constructible sheaves on $X$ for the given stratification $\cS$ and by $D^{b}_{c}(X)$ the bounded derived category of constructible complexes for this stratification. A constructible  sheaf, viewed as a complex concentrated in degree zero, is an object of $D^b_c(X)$, so that
\begin{equation*}
    \mathrm{Sh}_{c}(X)\hookrightarrow D^b_c(X)\subseteq D^b(X).
\end{equation*} Moreover, $D^b_c(X)$ for a given stratification $\cS$ of $X$ is a full triangulated subcategory of $D^{b}(X)$ \cite[Lemma 2.3.13 (1)]{Achar2021PerverseSheaves}. Moreover, $D^b_c(X)$ carries the (middle) perverse $t$-structure with respect to $\mathcal S$, which is bounded \cite[Theorem 3.1.9 (2)]{Achar2021PerverseSheaves}; its heart is the abelian category $\mathrm{Perv}_{\mathcal S}(X)$ of perverse sheaves constructible with respect to $\mathcal S$.

Thus a perverse sheaf is, in general, a complex of sheaves rather than a sheaf concentrated in degree zero. The heart $\mathrm{Perv}(X)$ is an abelian category, so it makes sense to consider its simple objects. These are described in terms of intermediate extensions of irreducible local systems on the strata and is exactly what we will use and focus later. More precisely, let $V\subset X$ be a smooth, connected, locally closed subset and $\cL$ an irreducible local system on $V$. The intersection cohomology complex $\mathrm{IC}(V,\cL)\in\mathrm{Perv}_{\mathcal S}(X)$, defined as in \cite[Definition 3.3.9]{Achar2021PerverseSheaves}, is a simple object of $\mathrm{Perv}_{\mathcal S}(X)$ \cite[Theorem 3.4.5 (1)]{Achar2021PerverseSheaves}.

In the particular case of a quoric manifold $M^{4n}$, we fix the face stratification, see Theorem~\ref{thm:quoric-stratified}. As shown in Lemma~\ref{lem:strata-simply-connected}, every local system on a face stratum is constant.

\begin{theorem}\label{thm:constructible-quoric} 
Let $M^{4n}$ be a quoric manifold with the face stratification $\{M_F\}_{F\in\cF(P^n)}$.
\begin{enumerate}
\item A sheaf $\cG$ is constructible if and only if, for every face $F$, $\cG|_{M_F}\cong\underline{V_F}_{M_F}$ is a constant sheaf with $V_F$ a finite-dimensional $\C$-vector space.
\item The simple objects of $\mathrm{Perv}_{\cS}(M^{4n})$ are, up to isomorphism, the complexes
\begin{equation*}
\mathrm{IC}_F:=\mathrm{IC}(M_F,\underline\C)=j_{F!*}\underline\C_{M_F}\big[\tfrac12\dim_{\mathbb R}M_F\big],
\qquad F\in\cF(P^n),
\end{equation*}
one for each face of $P^n$, and they are pairwise non-isomorphic.
\end{enumerate}
\end{theorem}

\begin{proof}
(1) Each face stratum $M_F$ is connected and simply connected by Lemma~\ref{lem:strata-simply-connected}. Hence every finite-rank local system on $M_F$ is constant. 

(2) Let $k=\mathrm{codim}F$. By \eqref{eq:orbit-type-slice}, and since $\lambda^b(F)\cong Q^k$ has dimension $3k$,
\begin{equation*}
    \dim_{\mathbb R}M_F=\dim\mathring F+\dim_{\mathbb R}Q^n/\lambda^b(F)=(n-k)+(3n-3k)=4\dim F.
\end{equation*}
In particular $\dim_{\mathbb R}M_F$ is even, so the shift $\tfrac12\dim_{\mathbb R}M_F=2\dim F$ in the definition of $\mathrm{IC}_F$ is an integer.  By Theorem~\ref{thm:quoric-stratified} and Lemma~\ref{lem:isotropy-of-strata}(1), $\{M_F\}_{F\in\cF(P^n)}$ is a finite stratification $\mathcal S$ of $M^{4n}$ by smooth, connected, locally closed submanifolds. Hence, by \cite[Theorem~3.4.5]{Achar2021PerverseSheaves}, the simple objects of $\mathrm{Perv}_{\mathcal S}(M^{4n})$ are the intermediate extensions of irreducible local systems on the strata. By Lemma~\ref{lem:strata-simply-connected}, $M_F$ is simply connected, so every local system on $M_F$ is constant and its only irreducible one is $\underline\C_{M_F}$, of rank one. Thus the simple objects are precisely the stated $\mathrm{IC}_F$, one for each face.

They are pairwise non-isomorphic because their supports are the distinct sets $\mathrm{supp}(\mathrm{IC}_F)=\overline{M_F}=\pi^{-1}(F)$.
\end{proof}

\section{The equivariant derived category of a quoric manifold}\label{sec:main-section}

Throughout this section, $M^{4n}$ denotes a quoric manifold over a simple polytope $P^n$ and is equipped with the finite face stratification
\begin{equation*}
M^{4n}=\bigsqcup_{F\in\cF(P^n)}M_F,
\qquad
M_F=\pi^{-1}(\mathring F).
\end{equation*}
The strata are smooth, connected, locally closed and $Q^n$-invariant. All constructibility conditions refer to this stratification. Perversity is normalised by half the real dimension of the strata, so that the shift on $M_F$ is $\tfrac12\dim_{\mathbb R}M_F=2\dim F$.

\subsection{Equivariant derived category of constructible sheaves}
\label{sec:equivariant-derived-cat-quoric}

In this section, we introduce the equivariant derived category of constructible sheaves on a quoric manifold. We proceed in two steps. Following \cite[Section 2]{Bernstein1994}, we first define the equivariant derived category of sheaves on the quoric manifold and then impose the constructibility condition. This is the main category of interest in the paper which we identify it with the derived category of constructible DG-modules over a DG-algebra defined on the poset of faces of the orbit space $P^n$.

Denote by $EQ^n=(S^\infty)^n$ the universal $Q^n$-bundle, a free $Q^n$-space. Let $M_Q=EQ^n\times_{Q^n}M^{4n}$ be the Borel construction of Section~\ref{sec:borel-construction}, and let $p\colon EQ^n\times M^{4n}\to M^{4n}$ be the projection, $p(s,m)=m$. Then $(EQ^n\times M^{4n},p)$ is a resolution of $M^{4n}$ in the sense of \cite[Definition 2.1.2]{Bernstein1994}, and we have the diagram
\begin{equation}\label{eq:resolution-diagram}
    M^{4n}\;\xleftarrow{\ p\ }\;EQ^n\times M^{4n}\;\xrightarrow{\ q\ }\;M_Q,
\end{equation}
where $q$ is the quotient by the free diagonal $Q^n$-action.

\begin{definition}\label{def:equivariant-derived-cat}
Let $M^{4n}$ be a quoric manifold and $(EQ^n\times M^{4n},p)$ the resolution above. The equivariant derived category $D^{b}_{Q^n}\big(M^{4n},\,EQ^n\times M^{4n}\big)$ of $M^{4n}$ is the category with:
\begin{description}
\item[Objects] triples $F=(F_{M^{4n}},\overline F,\beta)$ with $F_{M^{4n}}\in D^{b}(M^{4n})$, $\overline F\in D^{b}(M_Q)$, and $\beta$ an isomorphism $\beta\colon p^{*}F_{M^{4n}}\;\xrightarrow{\ \sim\ }\;q^{*}\overline F$ in $D^{b}(EQ^n\times M^{4n})$;
\item[Maps] from $F=(F_{M^{4n}},\overline F,\beta_F)$ to $G=(G_{M^{4n}},\overline G,\beta_G)$ are pairs $(f,\overline f)$ with $f\colon F_{M^{4n}}\to G_{M^{4n}}$ in $D^{b}(M^{4n})$ and $\overline f\colon\overline F\to\overline G$ in $D^{b}(M_Q)$, such that
\begin{equation}\label{square_for_maps}
\beta_G\circ p^{*}(f)\;=\;q^{*}(\overline f)\circ\beta_F.
\end{equation}
\end{description}
\end{definition}

By \cite[Section 2.7.2]{Bernstein1994} combined with \cite[Section 3.1]{Bernstein1994}, the resolution $(EQ^n\times M^{4n},p)$ is $\infty$-acyclic, see \cite[Definition 1.9.1]{Bernstein1994}. As a result the definition is independent of the resolution, and we write $D^{b}_{Q^n}(M^{4n}):=D^{b}_{Q^{n}}(M^{4n}, EQ^n\times M^{4n})$.

\begin{remark}\label{rem:not-a-localisation}
The fibre of $p$ is $EQ^n$, which is contractible, so an object of $D^{b}_{Q^n}(M^{4n})$ is essentially an object of $D^{b}(M_Q)$ whose pullback to $EQ^n\times M^{4n}$ is a pullback from $M^{4n}$. The abelian category $\mathrm{Sh}_{Q^n}(M^{4n})$ embeds in $D^{b}_{Q^n}(M^{4n})$ as the complexes concentrated in degree $0$. However, $D^{b}_{Q^n}(M^{4n})$ is \emph{not} obtained by localising the category of complexes over $\mathrm{Sh}_{Q^n}(M^{4n})$ at the quasi-isomorphisms; see \cite[Section 6.4]{Achar2021PerverseSheaves}.
\end{remark}

Now we consider quoric manifolds with the face stratification as shown in Theorem~\ref{thm:quoric-stratified}.

\begin{definition}\label{def:constructible-triples}
Let $D^{b}_{Q^n,c}(M^{4n})$ be the full subcategory of $D^{b}_{Q^n}(M^{4n})$ on those triples $(F_{M^{4n}},\overline F,\beta)$ for which $F_{M^{4n}}$ has cohomology sheaves constructible of finite rank with respect to the face stratification of $M^{4n}$.
\end{definition}

Via $(F_{M^{4n}},\overline F,\beta)\mapsto\overline F$, the category $D^{b}_{Q^n,c}(M^{4n})$ is identified with the full subcategory of $D^{b}(M_Q)$ consisting of complexes whose cohomology sheaves are constant of finite rank on each stratum $EQ^n\times_{Q^n}M_{F}$. This is the quoric analogue of the description used for toric varieties in \cite[Section 2.1]{Lunts1995}, and is the form in which we shall use the category. We record the resulting chain of inclusions:
\begin{equation*}
D^{b}_{Q^n,c}(M^{4n})\;\subseteq\;D^{b}_{Q^n}(M^{4n})\;\subseteq\;D^{b}(M_Q).
\end{equation*}

We close this section with generators of $D^{b}_{Q^n,c}(M^{4n})$. In the remaining, $D^b_c(M^{4n})$ and $D^b_{Q^n,c}(M^{4n})$ consist of objects constructible along the face stratification. Every stratum $M_F$ has real dimension $4\dim F$, so the function $\mathrm{\mathbf{p}}(M_F):=\tfrac12\dim_{\mathbb R}M_F=2\dim F$ is integer valued. We use the perverse t-structure on $D^b_c(M^{4n})$ attached to $p$: with $i_F\colon M_F\hookrightarrow M^{4n}$ the inclusion,
\begin{align*}
K\in{}^pD^{\le0}&\iff H^k(i_F^*K)=0\ \text{for all }k>-\mathrm{\mathbf{p}}(M_F),\\
K\in{}^pD^{\ge0}&\iff H^k(i_F^!K)=0\ \text{for all }k<-\mathrm{\mathbf{p}}(M_F),
\end{align*}
for every face $F$ \cite[Section~3.1]{Achar2021PerverseSheaves}. Its heart is $\mathrm{Perv}(M^{4n})$. On $D^b_{Q^n,c}(M^{4n})$ we use the t-structure for which $K$ lies in ${}^pD^{\le0}$ (resp.\ ${}^pD^{\ge0}$) iff its underlying object $K_{M^{4n}}$ does \cite[Section~5.1]{Bernstein1994}; its heart is the category $\mathrm{Perv}_{Q^n}(M^{4n})$ of \emph{equivariant perverse sheaves}.

The category $\mathrm{Perv}_{Q^n}(M^{4n})$ is abelian, being the heart of a $t$-structure. It has finite length, because the forgetful functor to $\mathrm{Perv}(M^{4n})$ is exact and faithful and $\mathrm{Perv}(M^{4n})$ has finite length for a finite stratification with finite-rank local systems \cite[Theorem~3.3.1]{Achar2021PerverseSheaves}. Since the $t$-structure is bounded, an object of $D^b_{Q^n,c}(M^{4n})$ is a finite iterated cone of shifts of its perverse cohomology objects, each of which is a finite iterated extension of simple objects. Hence the simple objects generate $D^b_{Q^n,c}(M^{4n})$.

Moreover, by \cite[5.2]{Bernstein1994}, for a smooth, connected, locally closed, $Q^n$-invariant submanifold $j\colon V\hookrightarrow M^{4n}$ and a $Q^n$-equivariant local system $\cL$ on $V$, the intermediate extension $\mathrm{IC}_{Q^n}(\overline V,\cL)$ lies in $\mathrm{Perv}_{Q^n}(M^{4n})$,  and the simple objects of $\mathrm{Perv}_{Q^n}(M^{4n})$ are exactly the $\mathrm{IC}_{Q^n}(\overline V,\cL)$, where $V$ runs over smooth, connected, locally closed, $Q^n$-invariant submanifolds and $\cL$ over irreducible equivariant local systems on $V$. For a quoric manifold there are infinitely many such pairs $(V,\cL)$, for instance one for each orbit. Restricting to objects constructible along the face stratification, however, only the strata occur: a simple object $S$ has support $\overline{M_F}$ for a single face $F$, and its restriction to the open dense stratum $M_F$ of the support is an irreducible equivariant local system shifted by $\mathrm{\mathbf{p}}(M_F)=2\dim F$ \cite[Lemma 3.3.12]{Achar2021PerverseSheaves}, so $V=M_F$. By Proposition~\ref{prop:loc-on-face-strata} every $Q^n$-equivariant local system on $M_F$ is constant, so the only irreducible one is $\underline\C_{M_F}$, of rank one. Hence the simple objects are indexed by the faces of $P^n$. We write $j_F\colon M_F\hookrightarrow M^{4n}$ for the inclusion and $j_{F!*}$ for the intermediate extension to $\overline{M_F}$ followed by extension by zero to $M^{4n}$, and set $\mathrm{IC}_F:=j_{F!*}\big(\underline\C_{M_F}[2\dim F]\big)$. The above discussion is summarised in the following result.

\begin{proposition}\label{prop:simples}
Let $M^{4n}$ be a quoric manifold with the face stratification. Then, the following hold.
\begin{enumerate}
\item[(1)] $\mathrm{Perv}_{Q^n}(M^{4n})$ is an abelian category of finite length.
\item[(2)] Its simple objects are, up to isomorphism, the complexes
\begin{equation*}
\mathrm{IC}_F:=j_{F!*}\underline\C_{M_F}\big[\tfrac12\dim_{\mathbb R}M_F\big]
=\mathrm{IC}_{Q^n}(\overline{M_F},\underline\C),
\qquad F\in\cF(P^n),
\end{equation*}
one for each face of $P^n$, and these are pairwise non-isomorphic.
\item[(3)] $\{\mathrm{IC}_F\}_{F\in\cF(P^n)}$ generates $D^b_{Q^n,c}(M^{4n})$ as a triangulated category: every object is obtained from the $\mathrm{IC}_F$ by finitely many shifts and cones.
\end{enumerate}
\end{proposition}

\subsection{The DG-algebra}
\label{sec:quoric-dg-algebra}

In this section, we define a sheaf of DG-algebras on the finite space $\cF_{\preceq}$. It is the push-forward, along the projection $b_Q\colon M_Q\to\cF_{\preceq}$, of the de Rham sheaf on the Borel construction $M_Q$ of the quoric manifold $M^{4n}$. 

\begin{definition}
Let $W\subseteq M_{Q}$ be an open set of the Borel space. Define the de Rham sheaf as
\begin{equation*}
\Omega^{\bullet}_{M_Q}(W):=\varprojlim_N\,\Omega^{\bullet}\big(W\cap M_Q^N\big).
\end{equation*}
That is the limit of the de Rham complexes of the manifolds $W\cap M_Q^N$ along the restrictions to the closed submanifolds $W\cap M_Q^{N}\subset W\cap M_Q^{N+1}$.
\end{definition}

\begin{lemma}\label{lem:de-rham-properties}
For the de Rham sheaf $\Omega^{\bullet}_{M_Q}$ on the Borel space construction $M_Q$ of a quoric manifold, the following hold.
\begin{enumerate}
\item[(1)] $\Omega^{\bullet}_{M_Q}$ is a soft resolution of the constant sheaf $\mathbb{C}_{M_Q}$.
\item[(2)] For each $k \ge 0$, the sheaf $\Omega^{k}_{M_{Q}}$ is soft.
\item[(3)] Each sheaf $\Omega^{k}_{M_{Q}}$ is acyclic, i.e. $H^{i}\!\left(M_Q,\Omega^{k}_{M_{Q}}\right) = 0$, for all $i>0$.
\item[(4)] For any sheaf $S$ of $\C$-vector spaces on $M_Q$ and any $p\ge0$, the sheaf $\Omega^{p}_{M_Q}\otimes_\C S$ is soft, hence acyclic on every open $W\subseteq M_Q$.
\item[(5)] For any open subset $U \subset M_Q$, the complex $\Omega^{\bullet}_{M_Q}(U)$ is a supercommutative differential graded algebra under the wedge product.
\end{enumerate}
\end{lemma}

\begin{proof}
Throughout, $W\subseteq M_Q$ is an arbitrary open, paracompact set so Proposition~\ref{prop:soft-acyclic} applies on $W$.
\begin{enumerate}
\item[(1)] Every $x\in M_Q$ has a neighbourhood basis of connected open sets $W$ with $H^k(\Omega^\bullet_{M_Q}(W))=0$ for $k\ge1$ and $H^0=\C$. The stalk of the cohomology sheaf $\mathcal H^k(\Omega^\bullet_{M_Q})_x$ is the filtered colimit of these groups, so it vanishes for $k\ge1$ and equals $\C$ for $k=0$, the constants. Hence $0\to\underline\C\to\Omega^0_{M_Q}\to\Omega^1_{M_Q}\to\cdots$ is exact on stalks.
\item [(2)]  $\Omega^0_{M_Q}$ is soft by \cite[Theorem~12.2.4]{Bernstein1994}. Each $\Omega^k_{M_Q}$ is a module over the soft sheaf of rings $\Omega^0_{M_Q}$, hence soft by Proposition~\ref{prop:soft-acyclic}(2).
\item[(3)] $W$ is paracompact and $\Omega^k_{M_Q}|_W$ is soft by (2), so $H^i(W,\Omega^k_{M_Q})=0$ for $i>0$ by Proposition~\ref{prop:soft-acyclic}(1).
\item[(4)] The sheaf $\Omega^p_{M_Q}\otimes_\C S$ is a module over the soft sheaf of rings $\Omega^0_{M_Q}$ via the first factor, so it is soft by Proposition~\ref{prop:soft-acyclic}(2), and acyclic on every open $W$ by Proposition~\ref{prop:soft-acyclic}(1), since $W$ is paracompact.
\item[(5)] Wedge products and $d$ commute with restriction to $W^N\subseteq W^{N+1}$, and limits of supercommutative DG-algebras are computed termwise, so $\Omega^\bullet_{M_Q}(W)$ is a supercommutative DG-algebra, and the restriction maps are algebra maps.\qedhere
\end{enumerate}
\end{proof}

We are ready to define the sheaf of $DG$-algebras over $\cF_{\preceq}$, the Alexandrov space associated to the face poset.

\begin{definition}\label{def:cA}
Let $\cF_{\preceq}$ be the Alexandrov space associated to face poset $(\cF(P^n), \preceq)$ and $b_Q\colon M_Q \to \cF_{\preceq}$ the projection map. Define the sheaf $\cA$ of DG-algebras over $\cF_{\preceq}$ as
\begin{equation}\label{eq:sheaf-over-poset}
    \cA:=b_{Q,*}\big(\Omega^{\bullet}_{M_Q}\big).
\end{equation}
For a face $F$ with minimal open neighbourhood $U_F$, the stalk $\cA_{F}$ is defined by
\begin{equation}
    \cA_{F}
    = \cA(U_F)
    = \Omega^{\bullet}_{M_Q}(b_Q^{-1} U_F)
    = \Omega^{\bullet}_{M_Q}(V_{F,Q}).
\end{equation}
\end{definition}

\begin{remark}
The functor $b_{Q,*}$ in \eqref{eq:sheaf-over-poset} need not be derived as $\Omega^{\bullet}_{M_Q}$ consists of soft sheaves and therefore is acyclic, see Lemma~\ref{lem:de-rham-properties}.
\end{remark}

\begin{remark}
The sheaf $\cA$ is a sheaf of graded-commutative DG-algebras on $\cF_\preceq$. Indeed, $\Omega^{\bullet}_{M_Q}$ is a sheaf of DG-algebras, and the push-forward functor $b_{Q,*}$ is defined section-wise, and it preserves the differential graded algebra structure. Hence $b_{Q,*}\bigl(\Omega^{\bullet}_{M_Q}\bigr)$ is naturally a sheaf of DG-algebras on $\cF_{\preceq}$.
\end{remark}

The stalks of $\cA$ record the equivariant cohomology of the strata, which is invisible to $\mathrm{Sh}_{Q^n}(M^{4n})$, see Corollary~\ref{prop:sheaves-on-orbits}.

\begin{proposition}\label{prop:stalks-of-cA}
For $F\in\cF(P^n)$ there is an isomorphism of graded algebras
\begin{equation*}
    H^{*}(\cA_F)
    \cong \C[x_1,\dots,x_k],
    \qquad k=\mathrm{codim}(F),\ \deg x_i=4.
\end{equation*}
In particular $H^{*}(\cA_F)\cong\C$, concentrated in degree $0$, exactly when $F=P^n$; at every proper face the stalk is a nontrivial polynomial ring.
\end{proposition}

\begin{proof}
By definition, $\cA_F=\Omega^\bullet_{M_Q}(V_{F,Q})$. The open set $V_{F,Q}$ is paracompact, and by Lemma~\ref{lem:de-rham-properties}, $\Omega^\bullet_{M_Q}|_{V_{F,Q}}$ is a soft resolution of $\underline\C_{V_{F,Q}}$. By the abstract de Rham theorem (Theorem~\ref{thm:abstract-de-rham}),
\begin{equation*}
    H^*(\cA_F)=H^*\big(\Omega^\bullet_{M_Q}(V_{F,Q})\big)\cong H^*(V_{F,Q};\C).
\end{equation*}
By Proposition~\ref{prop:borel-deformation-retract}, $V_{F,Q}\simeq B\lambda^b(F)$, and cohomology with constant coefficients is homotopy invariant on paracompact spaces \cite[II.11.12]{bredon_sheaf_1997}. Hence
\begin{equation*}
    H^*(\cA_F)\cong H^*(B\lambda^b(F);\C).
\end{equation*}
Now $\lambda^b(F)\cong Q^k$ with $k=\mathrm{codim}F$, and $H^*(BQ;\C)=H^*(B\mathrm{SU}(2);\C)=\C[c_2]$ with $\deg c_2=4$. The Künneth formula gives
\begin{equation*}
    H^*(BQ^k;\C)\cong\C[c_2]^{\otimes k}=\C[x_1,\dots,x_k],\qquad \deg x_i=4.
\end{equation*}
The cup product on $H^*(V_{F,Q};\C)$ corresponds to the wedge product on $\Omega^\bullet_{M_Q}(V_{F,Q})$, since the de Rham isomorphism is multiplicative, so this is an isomorphism of graded algebras.

For $F=P^n$ we have $k=0$, so $H^*(\cA_F)=\C$ in degree $0$. For a proper face $k\ge1$, and $\C[x_1,\dots,x_k]$ is nonzero in every degree divisible by $4$, hence a nontrivial polynomial ring.
\end{proof}
\subsection{Derived category of constructible DG-modules}
Let $\cA$ be the sheaf of DG-algebras from Section~\ref{sec:quoric-dg-algebra}. We write $\cA\text{-mod}$ for the abelian category of DG-$\cA$-modules and $D_\cA:=D(\cA)$ for its derived category, as recalled in Section~\ref{sec:derived-cats}. The sheaf $\cA$ lives on the orbit space $M^{4n}/Q^n$, which is identified with the simple polytope $P^n$ by Definition~\ref{def:quoric}. We regard $\cA$ as a sheaf on the finite space $\cF_\preceq$, namely the face poset $(\cF(P^n),\preceq)$ with the Alexandrov topology, see Section~\ref{sec:finite-spaces}. For $U\subseteq\cF_\preceq$ open, let $j_U\colon U\hookrightarrow\cF_\preceq$ be the inclusion map and define
\begin{itemize}
\item The \emph{restriction} of a sheaf $\mathcal G$ on $\cF_\preceq$ to $U$ is the sheaf $\mathcal G|_U=j_U^{-1}\mathcal G$ on $U$ given by $\mathcal G|_U(V)=\mathcal G(V)$ for every open $V\subseteq U$. Its stalks are $(\mathcal G|_U)_F=\mathcal G_F$ for $F\in U$. Restriction is applied to sheaves of DG-algebras and DG-modules termwise.
\item The \emph{extension by zero} of a sheaf $\mathcal H$ on $U$ is the sheaf $(j_U)_!\mathcal H$ on $\cF_\preceq$ obtained by sheafifying the presheaf $V\mapsto\mathcal H(V)$ if $V\subseteq U$ and $V\mapsto 0$ otherwise. Its stalks are
\begin{equation*}
\big((j_U)_!\mathcal H\big)_F=
\begin{cases}\mathcal H_F & F\in U,\\ 0 & F\notin U.\end{cases}
\end{equation*}
\end{itemize}
Since $\cF_\preceq$ is finite, a sheaf on it is the same as a functor on the poset $\cF(P^n)$, with $\mathcal G_F=\mathcal G(U_F)$, so both operations can also be read off from stalks and restriction maps. The functor $(j_U)_!$ is exact. We put
\begin{equation*}
\cA_U:=(j_U)_!\,(\cA|_U),\qquad \cA_F:=\cA_{U_F}\quad (F\in\cF(P^n)).
\end{equation*}
Thus $\cA_U$ is a sheaf of DG-$\cA$-modules on all of $\cF_\preceq$ with stalks $\cA_G$ for $G\in U$ and $0$ for $G\notin U$.

Following \cite[(1.6.2)]{Lunts1995}, we have the following definition of the constructible subcategory.
\begin{definition}\label{def:const-subcat-D_A}
The \emph{constructible subcategory} $D_{\cA,c}\subseteq D_\cA$ is the full triangulated subcategory generated by the objects $\cA_U$, $U\subseteq\cF_\preceq$ open. 
\end{definition}
 
That is, it is the smallest full subcategory containing them that is closed under shifts, cones and isomorphisms.

Next, we define the functors between $D^{+}(M_Q)$, the bounded-below derived category of sheaves on $M_Q$, and $D_\cA$, the unbounded derived category of $\cA$-modules.

\begin{definition}\label{def:global-sections-functor}
Let $M^{4n}$ be a quoric manifold and $M_Q$ its homotopy quotient.
\begin{itemize}
\item[(1)] The \emph{global sections functor} $\Gamma\colon D^{+}(M_Q)\to D_{\cA}$ maps $\cG^{\bullet} \in D^{+}(M_Q)$ to
\begin{equation*}
\Gamma(\cG^{\bullet}):=b_{Q,*}\big(\Omega^{\bullet}_{M_Q}\otimes_{\C}\cG^{\bullet}\big).
\end{equation*}
\item[(2)] The \emph{localisation functor} $L\colon D_{\cA}\to D(M_Q)$ maps $\mathcal N\in D_\cA$ to
\begin{equation*}
L(\mathcal N):=\Omega^{\bullet}_{M_Q}\otimes_{b_Q^{-1}\cA}b_Q^{-1}\mathcal N',
\end{equation*}
where $\mathcal N'\to\mathcal N$ is a $\mathcal K$-flat resolution, as in \cite[1.7.4]{Lunts1995} and \cite[10.12.2.5, 10.12.2.6, 10.12.4.3]{Bernstein1994}.
\end{itemize}
\end{definition}

Note that $\Gamma(\cG^{\bullet})$ is a well-defined $\cA$-module. Indeed, $\Omega^{\bullet}_{M_Q} \otimes_{\C} \cG^{\bullet}$ is a complex of soft sheaves (Lemma~\ref{lem:de-rham-properties}), it is an $\Omega^{\bullet}_{M_Q}$-module by extension of scalars from $\cG^\bullet$, and there is a quasi-isomorphism $\Omega^{\bullet}_{M_Q} \otimes_{\C} \cG^{\bullet}\simeq\cG^{\bullet}$. 

The localisation functor $L$ is also well defined as  the existence of the $\mathcal K$-flat resolution $\mathcal N'$ -- a variant of the bar resolution -- means the underived tensor product $\Omega^\bullet_{M_Q}\otimes_{b_Q^{-1}\cA}b_Q^{-1}\mathcal N'$ computes the derived tensor product, independently of the choice of resolution up to quasi-isomorphism. 

When $\Gamma$ and $L$ are restricted to the constructible subcategories consist quasi-inverses to each other. In particular, we have the following.

\begin{theorem}\label{thm:main-theorem}
Let $M^{4n}$ be a quoric manifold. Then the functors $\Gamma\colon D^{+}(M^{4n})\to D_{\cA}$ and $L\colon D_{\cA}\to D^{+}(M^{4n})$
\begin{enumerate}
\item[(1)] preserve the subcategories $D^{b}_{Q^n,c}(M^{4n})$ and $D_{\cA,c}$ of constructible objects;
\item[(2)] induce mutually quasi-inverse equivalences
\begin{equation*}
D^{b}_{Q^n,c}(M^{4n}) \xrightarrow{\ \sim\ } D_{\cA,c}.
\end{equation*}
\end{enumerate}
\end{theorem}
It is worth mentioning here, that although the functors are defined on the level of unbounded derived categories, they are inverse to each other when restricted to the full subcategories of the constructible objects.

The proof requires an analysis of the families of generators of each constructible subcategory, which is carried out in Section~\ref{sec:generators}. The proof of the theorem is then given in Section~\ref{sec:proof-of-equiv}.

\subsection{Generators and the unit and counit}\label{sec:generators}

We first show that $\{\cA_{U_F}\}_F$ generates $D_{\cA,c}$ and that $\{j_{F!}\underline{\C}_{V_{F,Q}}\}_F$ generates $D^b_{Q^n,c}(M^{4n})$, where $V_{F,Q}=b_Q^{-1}(U_F)$ is the open star of $M_F$ in $M_Q$ and $j_F\colon V_{F,Q}\hookrightarrow M_Q$ is the open embedding. We then show that $\Gamma(j_{F!}\underline{\C}_{V_{F,Q}})\simeq\cA_{U_F}$ and $L(\cA_{U_F})\simeq j_{F!}\underline{\C}_{V_{F,Q}}$ for every $F$.

\begin{lemma}\label{lem:generation-by-minimal-opens}
For $F\in\cF(P^n)$ let $U_F$ be its minimal open neighbourhood in the Alexandrov topology on $\cF(P^n)$. Then $D_{\cA,c}$ is generated as a triangulated category by $\{\cA_{U_F}\}_{F\in\cF(P^n)}$, where $\cA_U$ denotes the extension by zero of $\cA|_U$.
\end{lemma}

\begin{proof}
Let $\mathcal T$ be the triangulated subcategory generated by the $\cA_{U_F}$. We show $\cA_U\in\mathcal T$ for every open $U$ by induction on the number $|U|$ of points of $U$. If $U=\emptyset$, then $\cA_U=0$, and if $U=U_F$ for some $F$, there is nothing to prove. Otherwise, since $U=\bigcup_{G\in U}U_G$, we may choose a closed point $F$ of the finite space $U$, that is, $F\in U$ with $F\notin U_G$ for all $G\in U\setminus\{F\}$, and put $U_1=U_F$, $U_2=\bigcup_{G\in U\setminus\{F\}}U_G$. Then $U_1\cup U_2=U$, $U_1\subsetneq U$ as $U\neq U_F$, and $U_2\subsetneq U$ as $F\notin U_2$.

For an open set $V$, write $\cA_V=j_{V!}(\cA|_V)$ for the extension by zero of the restriction of $\cA$ to $V$, so that $(\cA_V)_G=\cA_G$ if $G\in V$ and $(\cA_V)_G=0$ otherwise. Consider the canonical sequence
\begin{equation*}
0\to\cA_{U_1\cap U_2}\to\cA_{U_1}\oplus\cA_{U_2}\to\cA_{U}\to0
\end{equation*}
built from the inclusions of extensions by zero. It is exact, which we check on stalks, using that $(\cA_V)_G=\cA_G$ if $G\in V$ and $0$ otherwise:
\begin{itemize}
\item for $G\in U_1\cap U_2$, the stalk sequence is $0\to\cA_G\to\cA_G\oplus\cA_G\to\cA_G\to0$;
\item for $G\in U_1\setminus U_2$, it is $0\to0\to\cA_G\to\cA_G\to0$, and for $G\in U_2\setminus U_1$ it is the same with $\cA_G$ in the second summand;
\item for $G\notin U=U_1\cup U_2$, all terms vanish.
\end{itemize}
Hence $\cA_U$ is the cone of $\cA_{U_1\cap U_2}\to\cA_{U_1}\oplus\cA_{U_2}$ in $D_\cA$. The sets $U_1,U_2,U_1\cap U_2$ are open with fewer points than $U$, so by induction $\cA_{U_1},\cA_{U_2},\cA_{U_1\cap U_2}\in\mathcal T$, and therefore $\cA_U\in\mathcal T$. Since $D_{\cA,c}$ is generated by the $\cA_U$ for $U$ open, Definition~\ref{def:const-subcat-D_A}, we conclude $\mathcal T=D_{\cA,c}$.
\end{proof}

For an open set $U\subseteq\cF_\preceq$ write $V_{U,Q}:=b_Q^{-1}(U)$ and let $j_U\colon V_{U,Q}\hookrightarrow M_Q$ be the open embedding. Set $j_F:=j_{U_F}$, so that $V_{U_F,Q}=V_{F,Q}$.  We use the identification of $D^b_{Q^n,c}(M^{4n})$ with a full triangulated subcategory $\mathcal K\subseteq D^b(M_Q)$ from Section~\ref{sec:equivariant-derived-cat-quoric}. For $K\in\mathcal K$ and a face $F$ let $K|_F:=K|_{M_{F,Q}}$ and let 
\begin{equation*}
\mathrm{supp}(K):=\{F\in\cF(P^n): K|_F\neq 0\}.
\end{equation*} be the support.

\begin{lemma}\label{lem:generation-sheaf-side} $\{j_{F!}\underline\C_{V_{F,Q}}\}_{F\in\cF(P^n)}$ generates $D^b_{Q^n,c}(M^{4n})$ as a triangulated category.
\end{lemma}
\begin{proof}
Let $\mathcal T\subseteq D^b(M_Q)$ be the smallest full replete triangulated subcategory containing $j_{F!}\underline\C_{V_{F,Q}}$ for all $F\in\cF(P^n)$. Since $\mathcal K$ is a triangulated subcategory containing these generators, $\mathcal T\subseteq\mathcal K$. For the other inclusion we proceed in three steps. For an open set $U\subseteq\cF_\preceq$ we write
\begin{equation*}
E_U:=j_{U!}\underline\C_{V_{U,Q}},\qquad V_{U,Q}=b_Q^{-1}\bigl(U\bigr),
\end{equation*}
so that $E_{U_F}=j_{F!}\underline\C_{V_{F,Q}}$ is one of the generators of $\mathcal T$.

\medskip
\noindent\textbf{Step 1: $E_U\in\mathcal T$ for every open $U\subseteq\cF_\preceq$.}
We use induction on the number $|U|$ of points of $U$. If $U=\emptyset$ then $E_U=0$, and if $U=U_F$ for some $F$ then $E_U$ is a generator, so there is nothing to prove in these cases. Otherwise, choose $F\in U$ minimal with respect to inclusion, that is, such that $F\notin U_G$ for every $G\in U\setminus\{F\}$, and put
\begin{equation*}
U_1=U_F,\qquad U_2=\bigcup_{G\in U\setminus\{F\}}U_G.
\end{equation*}
Both are open, and $U_1\cup U_2=U$ because $U=\bigcup_{G\in U}U_G$. Moreover $U_1\subseteq U$, as $U$ is open and contains $F$, and $U_1\neq U$ by assumption, hence $U_1\subsetneq U$. Likewise $U_2\subseteq U$, and $F\in U\setminus U_2$, hence $U_2\subsetneq U$. Consequently $U_1\cap U_2$ is open and contained in $U_1\subsetneq U$.

Since $V_{U_1\cap U_2,Q}=V_{U_1,Q}\cap V_{U_2,Q}$ and $V_{U,Q}=V_{U_1,Q}\cup V_{U_2,Q}$, there is a sequence
\begin{equation*}
0\to E_{U_1\cap U_2}\to E_{U_1}\oplus E_{U_2}\to E_U\to 0
\end{equation*}
on $M_Q$, built from the inclusions of open sets. It is exact, which we check on stalks at a point $x\in M_Q$. Using that the stalk of $E_V$ at $x$ is $\C$ if $b_Q(x)\in V$ and $0$ otherwise, the stalk sequence is
\begin{itemize}
\item $0\to\C\to\C^2\to\C\to0$ if $b_Q(x)\in U_1\cap U_2$;
\item $0\to0\to\C\to\C\to0$ if $b_Q(x)$ lies in exactly one of $U_1,U_2$, with $\C$ in the first or second summand accordingly;
\item the zero sequence if $b_Q(x)\notin U$.
\end{itemize}
Hence the sequence is exact, and it gives a distinguished triangle in $D^b(M_Q)$. The sets $U_1$, $U_2$ and $U_1\cap U_2$ are open with fewer points than $U$, so by the induction hypothesis $E_{U_1}$, $E_{U_2}$ and $E_{U_1\cap U_2}$ lie in $\mathcal T$. Since $\mathcal T$ is closed under finite sums and cones, $E_U\in\mathcal T$.

\medskip
\noindent\textbf{Step 2: $j_{F!}i_*\underline\C_{M_{F,Q}}\in\mathcal T$ for every face $F$, where $i\colon M_{F,Q}\hookrightarrow V_{F,Q}$.}
Fix $F$ and put $U'=U_F\setminus\{F\}$. Because $F$ is the minimum of $U_F$, the set $U'$ is open. Then $V_{U',Q}=V_{F,Q}\setminus M_{F,Q}$, so $V_{U',Q}$ is open in $V_{F,Q}$ and $M_{F,Q}$ is its closed complement. Therefore $j_{F!}i_*\underline\C_{M_{F,Q}}$ is the constant sheaf $\underline\C$ on $M_{F,Q}$ extended by zero to $M_Q$, and the sequence
\begin{equation*}
0\to E_{U'}\to j_{F!}\underline\C_{V_{F,Q}}\to j_{F!}i_*\underline\C_{M_{F,Q}}\to 0
\end{equation*}
is the extension by zero of the standard sequence for an open subset and its closed complement. It is exact, as the stalks are $\C\to\C\to0$ on $V_{U',Q}$, $0\to\C\to\C$ on $M_{F,Q}$, and zero elsewhere. By Step 1 the term $E_{U'}$ lies in $\mathcal T$, and the middle term is a generator. Hence the third term lies in $\mathcal T$ as well.

\medskip
\noindent\textbf{Step 3: every $K\in\mathcal K$ lies in $\mathcal T$.}
For $K\in D^b(M_Q)$ let $\operatorname{supp}K$ be the set of faces $G$ with $K|_{M_{G,Q}}\neq0$. We use induction on $|\operatorname{supp}K|$. If $\operatorname{supp}K=\emptyset$ then $K=0$, because the strata $M_{G,Q}$ cover $M_Q$. Otherwise choose $F\in\operatorname{supp}K$ maximal with respect to inclusion, and put $j=j_F$, $V=V_{F,Q}$. The adjunction triangle
\begin{equation*}
j_!j^*K\to K\to K'\to
\end{equation*}
has its first two terms in $\mathcal K$, and hence so does $K'$.

We compute the support of $K'$ stratum by stratum. For $G\in U_F$ the stratum $M_{G,Q}$ lies in $V$, so $j_!j^*K|_{M_{G,Q}}\to K|_{M_{G,Q}}$ is an isomorphism and $K'|_{M_{G,Q}}=0$. For $G\notin U_F$ we have $j_!j^*K|_{M_{G,Q}}=0$, so $K'|_{M_{G,Q}}\cong K|_{M_{G,Q}}$. Hence
\begin{equation*}
\operatorname{supp}K'=\operatorname{supp}K\setminus U_F,
\end{equation*}
which is strictly smaller than $\operatorname{supp}K$ since $F\in U_F\cap\operatorname{supp}K$. By induction $K'\in\mathcal T$, so it suffices to show $j_!j^*K\in\mathcal T$.

Put $L=j^*K$. The open set $V\setminus M_{F,Q}$ is the union of the strata $M_{G,Q}$ with $G\supsetneq F$, and by maximality of $F$ no such $G$ lies in $\operatorname{supp}K$. Hence $L$ vanishes on $V\setminus M_{F,Q}$. For the open embedding $j'\colon V\setminus M_{F,Q}\hookrightarrow V$ the triangle
\begin{equation*}
j'_!j'^*L\to L\to i_*i^*L\to
\end{equation*}
has $j'^*L=0$, so $L\cong i_*i^*L$. Writing $L_F:=K|_{M_{F,Q}}=i^*L$, we obtain
\begin{equation*}
j_!j^*K\cong j_!i_*L_F.
\end{equation*}
Here $L_F\in D^b\bigl(M_{F,Q}\bigr)$ has bounded cohomology, and since $K\in\mathcal K$ and $M_{F,Q}$ is a single stratum, Lemma~\ref{lem:strata-simply-connected} and Proposition~\ref{prop:loc-on-face-strata} show that every cohomology sheaf $\mathcal H^m\left(L_F\right)\cong\underline\C^{\,r_m}_{M_{F,Q}}$ is constant of finite rank $r_m$.

For each $m$ there is a truncation triangle
\begin{equation*}
\tau_{<m}L_F\to\tau_{\le m}L_F\to\underline\C^{\,r_m}_{M_{F,Q}}[-m]\to.
\end{equation*}
Since only finitely many $\mathcal H^m\left(L_F\right)$ are nonzero, induction on their number expresses $L_F$ as an iterated cone of shifts of finite sums of $\underline\C_{M_{F,Q}}$. The functor $j_!i_*$ is triangulated and commutes with finite sums and shifts, so $j_!i_*L_F$ is an iterated cone of shifts of finite sums of $j_{F!}i_*\underline\C_{M_{F,Q}}$. These objects lie in $\mathcal T$ by Step 2, and $\mathcal T$ is closed under finite sums, shifts and cones. Hence $j_!j^*K\in\mathcal T$, and since $K'\in\mathcal T$ as well, the triangle above gives $K\in\mathcal T$. As a result $\mathcal K\subseteq\mathcal T$ and the proof is now complete.
\end{proof}

Now, having established that generating sets, we prove that the functors $\Gamma$ and $L$ exchange them.

\begin{proposition}\label{prop:gamma-on-generators}
For every $F\in\cF(P^n)$, the map
\begin{equation}\label{eq:qis-gamma-on-generators}
\cA_{U_F}\longrightarrow \Gamma\big(j_{F!}\underline\C_{V_{F,Q}}\big) = b_{Q,*}\big(\Omega^{\bullet}_{M_Q}\otimes_\C j_{F!}\underline\C_{V_{F,Q}}\big), \qquad 1\mapsto 1\otimes 1,
\end{equation}
is a quasi-isomorphism in $D_{\cA,c}$.
\end{proposition}

\begin{proof}
It suffices to check \eqref{eq:qis-gamma-on-generators} on stalks at each $F'\in\cF(P^n)$.

\emph{Case $F'\in U_F$.} Let $\iota\colon U_F\hookrightarrow\cF_\preceq$ be the inclusion. On the left, $\cA_{U_F}|_{U_F}=\cA|_{U_F}$ by definition of extension by zero. On the right, since $V_{F,Q}=b_Q^{-1}(U_F)$, restricting $j_{F!}\underline\C_{V_{F,Q}}$ to $b_Q^{-1}(U_F)$ gives $\underline\C_{V_{F,Q}}$, so
\begin{equation*}
b_{Q,*}\big(\Omega^\bullet_{M_Q}\otimes_\C j_{F!}\underline\C_{V_{F,Q}}\big)\Big|_{U_F}
=\iota^{-1}b_{Q,*}\big(\Omega^\bullet_{M_Q}\otimes_\C\underline\C_{V_{F,Q}}\big)
=\iota^{-1}b_{Q,*}\big(\Omega^\bullet_{M_Q}\big)
=\cA|_{U_F},
\end{equation*}
using $\Omega^\bullet_{M_Q}\otimes_\C\underline\C_{V_{F,Q}}=\Omega^\bullet_{M_Q}|_{V_{F,Q}}$. The two sides agree.

\emph{Case $F'\notin U_F$.} Then $(\cA_{U_F})_{F'}=0$. It remains to show that
$b_{Q,*}\big(\Omega^{\bullet}_{M_Q}\otimes_\C j_{F!}\underline\C_{V_{F,Q}}\big)_{F'}$ is acyclic. Since $\cF_{\preceq}$ is a finite space, the stalk at $F'$ is the value on the minimal open neighborhood $U_{F'}$, and $b_Q^{-1}(U_{F'})=V_{F',Q}$. Using the softness of $\Omega^\bullet_{M_Q}\otimes_\C j_{F!}\underline\C_{V_{F,Q}}$ (it is an $\Omega^0_{M_Q}$-module), its sections compute hypercohomology on $V_{F',Q}$, and it is quasi-isomorphic to $j_{F!}\underline\C_{V_{F,Q}}$ by the Poincar\'e lemma. Hence
\begin{align*}
b_{Q,*}\left(\Omega^{\bullet}_{M_Q}\otimes_\C j_{F!}\underline\C_{V_{F,Q}}\right)_{F'}
&=\Gamma\left(V_{F',Q},\,\Omega^{\bullet}_{M_Q}\otimes_\C j_{F!}\underline\C_{V_{F,Q}}\right)\\
&\simeq R\Gamma\left(V_{F',Q},\, j_{F!}\underline\C_{V_{F,Q}}\right).
\end{align*}

Put $W:=V_{F',Q}$, $O:=W\cap V_{F,Q}$, which is open in $W$, and $Z:=W\setminus O$, which is closed. Then $(j_{F!}\underline\C_{V_{F,Q}})|_W=(j_O)_!\underline\C_O$, and the exact sequence
\begin{equation*}
0\to(j_O)_!\underline\C_O\to\underline\C_W\to\underline\C_Z\to 0
\end{equation*}
gives a long exact sequence
\begin{equation*}
\cdots\to H^i\big(W,(j_O)_!\underline\C_O\big)\to H^i(W;\C)\xrightarrow{\ \mathrm{res}\ }H^i(Z;\C)\to H^{i+1}\big(W,(j_O)_!\underline\C_O\big)\to\cdots.
\end{equation*}
It therefore suffices to show that $\mathrm{res}$ is an isomorphism. Since $W$ and $Z$ are paracompact and cohomology with constant coefficients is homotopy invariant \cite[II.11.12]{bredon_sheaf_1997}, it is enough that $Z\hookrightarrow W$ is a homotopy equivalence.

The set $Z$ is the union of the strata $EQ^n\times_{Q^n}M_G$ with $G\supseteq F'$ and $G\not\supseteq F$. In particular $M_{F',Q}\subseteq Z$, because $F'\not\supseteq F$. Let $H_Q$ be the homotopy of Proposition~\ref{prop:borel-deformation-retract} for $F'$, and fix $f'\in\mathring F'$. For such $G$, each $H_{Q,t}$ preserves the closure of $EQ^n\times_{Q^n}M_G$ in $W$, which is the union of the strata $EQ^n\times_{Q^n}M_{G''}$ with $F'\subseteq G''\subseteq G$. Since $G\not\supseteq F$ forces $G''\not\supseteq F$, this closure lies in $Z$. Hence $H_Q$ restricts to a homotopy of $Z$ that ends in $EQ^n\times_{Q^n}O_{f'}\subseteq M_{F',Q}\subseteq Z$ and fixes it pointwise. So both $W$ and $Z$ deformation retract onto $EQ^n\times_{Q^n}O_{f'}$, and $Z\hookrightarrow W$ is a homotopy equivalence. Thus $H^\ast(W,(j_O)_!\underline\C_O)=0$, and both stalks agree.
\end{proof}

\begin{proposition}\label{prop:L-on-generators}
For every $F\in\cF(P^n)$, the map
\begin{equation}\label{eq:qis-L-on-generators}
j_{F!}\underline\C_{V_{F,Q}}\longrightarrow L(\cA_{U_F})=\Omega^{\bullet}_{M_Q}\otimes_{b_Q^{-1}\cA}b_Q^{-1}\cA_{U_F}, \qquad s\mapsto 1\otimes s,
\end{equation}
is a quasi-isomorphism in $D(M_Q)$, hence in $D^{b}_{Q^n,c}(M^{4n})$.
\end{proposition}

\begin{proof}
The map is well defined: both sheaves are supported on $V_{F,Q}=b_Q^{-1}(U_F)$, and there $b_Q^{-1}\cA_{U_F}=b_Q^{-1}(\cA|_{U_F})$ contains the constants $\underline\C$. We check stalks at $x\in V_{F,Q}$ and $x\notin V_{F,Q}$ separately.

\emph{Case $x\in V_{F,Q}$.} Here $(j_{F!}\underline\C_{V_{F,Q}})_x=\C$, while
\begin{equation*}
\big(\Omega^\bullet_{M_Q}\otimes_{b_Q^{-1}\cA}b_Q^{-1}\cA_{U_F}\big)_x
\simeq \Omega^{\bullet}_{M_Q,x}\otimes_{\cA_{b_Q(x)}}(\cA_{U_F})_{b_Q(x)}
=\Omega^{\bullet}_{M_Q,x}\otimes_{\cA_{b_Q(x)}}\cA_{b_Q(x)}
\simeq\Omega^{\bullet}_{M_Q,x}\simeq\C,
\end{equation*}
using $b_Q(x)\in U_F$ (so $(\cA_{U_F})_{b_Q(x)}=\cA_{b_Q(x)}$) and the Poincar\'e lemma (Lemma~\ref{lem:de-rham-properties}(1)).

\emph{Case $x\notin V_{F,Q}$.} Then $(j_{F!}\underline\C_{V_{F,Q}})_x=0$; also $b_Q(x)\notin U_F$, so $(\cA_{U_F})_{b_Q(x)}=0$ and hence $(b_Q^{-1}\cA_{U_F})_x=0$, so the right-hand stalk vanishes too.
\end{proof}

\subsection{Proof of Theorem~\ref{thm:main-theorem}} \label{sec:proof-of-equiv}

We now combine the preceding results to prove Theorem~\ref{thm:main-theorem}. Throughout, we identify $D^b_{Q^n,c}(M^{4n})$ with the full triangulated subcategory $\mathcal K\subseteq D^b(M_Q)$. We also use the following standard facts about DG-modules over the sheaf of DG-algebras $\cA$ on the finite space $\cF_\preceq$. They are established in \cite[Section~1]{Lunts1995}, see also \cite[Section~10.12]{Bernstein1994}, and we do not reprove them here.

\begin{enumerate}
\item[(R1)] The modules $\cA_{U_F}$ are K-projective, and $b_Q^{-1}\cA_{U_F}$ is K-flat over $b_Q^{-1}\cA$.
\item[(R2)] Every DG-$\cA$-module $\cN$ has a functorial resolution $r_{\cN}\colon R(\cN)\to\cN$ by a K-projective DG-$\cA$-module, and $b_Q^{-1}R(\cN)$ is K-flat over $b_Q^{-1}\cA$.
\item[(R3)] For every DG-$\Omega^\bullet_{M_Q}$-module $\cM$, the action map $\theta_{\cM}\colon\Omega^\bullet_{M_Q}\otimes_\C\cM\to\cM$ is a quasi-isomorphism. If $\cM$ is bounded below, $b_{Q,*}(\theta_{\cM})$ is also a quasi-isomorphism.
\end{enumerate}

Setting $L(\cN):=\Omega^\bullet_{M_Q}\otimes_{b_Q^{-1}\cA}b_Q^{-1}R(\cN)$, we have $\Gamma L(\cN)=b_{Q,*}\big(\Omega^\bullet_{M_Q}\otimes_\C L(\cN)\big)$. The unit and counit of the adjunction are defined as follows.

For a DG-$\cA$-module $\cN$ let
\begin{equation*}
\eta_{\cN}\colon R(\cN)\to b_{Q,*}L(\cN),\qquad n\mapsto 1\otimes n,
\end{equation*}
and using (R3), define the \emph{unit}
\begin{equation*}
u_{\cN}:=b_{Q,*}(\theta_{L(\cN)})^{-1}\circ\eta_{\cN}\circ r_{\cN}^{-1}\colon\ \cN\longrightarrow\Gamma L(\cN).
\end{equation*} in $D_{\cA}$.

For $E\in D^+(M_Q)$ let $r\colon R\to\Gamma(E)$ be a projective resolution of $\Gamma(E)=b_{Q,*}\big(\Omega^\bullet_{M_Q}\otimes_\C E\big)$ as a DG-$\cA$-module, and let the \emph{counit} $c_E\colon L\Gamma(E)\to E$ be the composite
\begin{equation*}
\Omega^\bullet_{M_Q}\otimes_{b_Q^{-1}\cA}b_Q^{-1}R\xrightarrow{\ \mathrm{id}\otimes b_Q^{-1}r\ }\Omega^\bullet_{M_Q}\otimes_{b_Q^{-1}\cA}b_Q^{-1}b_{Q,*}\big(\Omega^\bullet_{M_Q}\otimes_\C E\big)\xrightarrow{\ m\ }\Omega^\bullet_{M_Q}\otimes_\C E\xleftarrow[\ \sim\ ]{\ e\mapsto1\otimes e\ }E,
\end{equation*}
where $m(\omega_1\otimes\omega_2\otimes e)=\omega_1\omega_2\otimes e$. As in \cite[Section~1]{Lunts1995}, $c_E$ does not depend on the choice of projective resolution up to homotopy, and $u$ and $c$ are natural transformations of triangulated functors.

\begin{lemma}\label{lem:u-c-on-generators}
For every $F\in\cF(P^n)$, put $\cG_F:=j_{F!}\underline{\C}_{V_{F,Q}}$. Then $u_{\cA_{U_F}}$ is an isomorphism in $D_\cA$ and $c_{\cG_F}$ is an isomorphism in $D^+(M_Q)$.
\end{lemma}

\begin{proof}
Let $\gamma\colon\cA_{U_F}\to\Gamma(\cG_F)$ be the map of Proposition~\ref{prop:gamma-on-generators} and $\mu\colon\cG_F\to L(\cA_{U_F})$ the map of Proposition~\ref{prop:L-on-generators}.

Since $\cA_{U_F}$ is projective by (R1), we may take $R(\cA_{U_F})=\cA_{U_F}$ and $r=\mathrm{id}$, so that $u_{\cA_{U_F}}=b_{Q,*}(\theta)^{-1}\circ\eta$ with $\theta=\theta_{L(\cA_{U_F})}$. By (R3), $b_{Q,*}(\theta)$ is an isomorphism, so it suffices to show that $\eta\colon\cA_{U_F}\to b_{Q,*}L(\cA_{U_F})$ is a quasi-isomorphism. It factors as
\begin{equation*}
\cA_{U_F}\xrightarrow{\ \gamma\ }b_{Q,*}\big(\Omega^\bullet_{M_Q}\otimes_\C\cG_F\big)\xrightarrow{\ b_{Q,*}(\mathrm{id}\otimes\mu)\ }b_{Q,*}\big(\Omega^\bullet_{M_Q}\otimes_\C L(\cA_{U_F})\big)\xrightarrow{\ b_{Q,*}\theta\ }b_{Q,*}L(\cA_{U_F}),
\end{equation*}
as one checks on the generator $1$. The first map is a quasi-isomorphism by Proposition~\ref{prop:gamma-on-generators}. The second is a quasi-isomorphism by Proposition~\ref{prop:L-on-generators}, since $\otimes_\C$ is exact and $b_{Q,*}$ preserves quasi-isomorphisms between the soft, bounded-below complexes involved. The third is a quasi-isomorphism by (R3). This proves that $u_{\cA_{U_F}}$ is an isomorphism in $D_\cA$.

We now treat $c_{\cG_F}$. Since $\cA_{U_F}$ is projective, we may use $\gamma$ as the projective resolution of $\Gamma(\cG_F)$ in the definition of $c_{\cG_F}$, so that $L\Gamma(\cG_F)=L(\cA_{U_F})$ and $c_{\cG_F}$ is the composite
\begin{equation*}
L(\cA_{U_F})\xrightarrow{\ \mathrm{id}\otimes b_Q^{-1}\gamma\ }\Omega^\bullet_{M_Q}\otimes_{b_Q^{-1}\cA}b_Q^{-1}b_{Q,*}\big(\Omega^\bullet_{M_Q}\otimes_\C\cG_F\big)\xrightarrow{\ m\ }\Omega^\bullet_{M_Q}\otimes_\C\cG_F\xleftarrow{\ \iota\ }\cG_F,
\end{equation*}
where $\iota\colon e\mapsto1\otimes e$ is a quasi-isomorphism. Since $\gamma(1)=1\otimes1$, the composite of $\mu$ with the first two maps is $s\mapsto1\otimes s$, that is, it equals $\iota$. Hence $c_{\cG_F}\circ\mu=\mathrm{id}_{\cG_F}$ in $D^+(M_Q)$. As $\mu$ is a quasi-isomorphism, $c_{\cG_F}=\mu^{-1}$ is an isomorphism in $D^+(M_Q)$.
\end{proof}

We are now ready to prove Theorem~\ref{thm:main-theorem}.

\begin{proof}[Proof of Theorem~\ref{thm:main-theorem}]
The functors $\Gamma$ and $L$ are triangulated, and $D_{\cA,c}$ and $\mathcal K$ are full triangulated subcategories closed under isomorphisms. 

By Lemma~\ref{lem:generation-sheaf-side}, $\mathcal K$ is generated by the objects $\cG_F$, and $\Gamma(\cG_F)\cong\cA_{U_F}\in D_{\cA,c}$ by Proposition~\ref{prop:gamma-on-generators}. Hence $\Gamma(\mathcal K)\subseteq D_{\cA,c}$.

Similarly, $D_{\cA,c}$ is generated by the objects $\cA_{U_F}$ by Lemma~\ref{lem:generation-by-minimal-opens}, and $L(\cA_{U_F})\cong\cG_F\in\mathcal K$ by Proposition~\ref{prop:L-on-generators}, so $L(D_{\cA,c})\subseteq\mathcal K$. Consequently $u$ restricts to a natural transformation $\mathrm{Id}_{D_{\cA,c}}\to\Gamma L$ of triangulated endofunctors of $D_{\cA,c}$, and $c$ restricts to a natural transformation $L\Gamma\to\mathrm{Id}_{\mathcal K}$.

Let $\mathcal N\subseteq D_{\cA,c}$ be the full subcategory of objects $N$ for which $u_N$ is an isomorphism. It is closed under shifts and isomorphisms. It is also closed under cones: given a distinguished triangle $N_1\to N_2\to N_3\to$ with $N_1,N_2\in\mathcal N$, the induced morphism of triangles has two isomorphic components, hence the third is an isomorphism by the five lemma. Thus $\mathcal N$ is a full triangulated subcategory containing all $\cA_{U_F}$ by Lemma~\ref{lem:u-c-on-generators}, so $\mathcal N=D_{\cA,c}$ by Lemma~\ref{lem:generation-by-minimal-opens}. The same argument applied to $\{E\in\mathcal K : c_E\text{ is an isomorphism}\}$, using Lemmas~\ref{lem:u-c-on-generators} and~\ref{lem:generation-sheaf-side}, shows that $c$ is an isomorphism on $\mathcal K$.

Therefore $\Gamma|_{\mathcal K}\colon\mathcal K\to D_{\cA,c}$ and $L|_{D_{\cA,c}}\colon D_{\cA,c}\to\mathcal K$ satisfy $\Gamma L\cong\mathrm{Id}$ and $L\Gamma\cong\mathrm{Id}$, so they are mutually quasi-inverse equivalences. Composing with the identification $\mathcal K=D^b_{Q^n,c}(M^{4n})$ gives the equivalence in the statement.
\end{proof}

\section{Formality of $\cA$}\label{sec:formality}

We work over $\R$ and tensor with $\C$ at the end. The construction follows Lunts \cite[Section~3]{Lunts1995}, and we make explicit the steps that are left implicit there. Let
\begin{equation*}
H^{\ast}(BQ^n;\R)=\R[u_1,\dots,u_n],\quad |u_i|=4,\qquad W:=\mathrm{span}\{u_1,\dots,u_n\}\subseteq H^{\ast}(BQ^n;\R).
\end{equation*}
Fix closed forms $\omega_i\in\Omega^4(BQ^n)$ representing the generators $u_i$, let $\pi\colon M_Q\to BQ^n$ be the classifying map, and put $\Omega_i:=\pi^*\omega_i$. Denote by
\begin{equation*}
\Omega:=\mathrm{span}_\R\{\Omega_1,\dots,\Omega_n\}\subseteq\Omega^4(M_Q)
\end{equation*}
the span of these closed global $4$-forms. For a face $F$ with minimal open neighbourhood $U_F=\{G:G\supseteq F\}$, recall from Definition~\ref{def:cA} that
\begin{equation*}
\cA_F:=\cA(U_F)=\Omega^{\bullet}_{M_Q}(b_Q^{-1}U_F)=\Omega^{\bullet}_{M_Q}(V_{F,Q}),
\end{equation*}
and for $G\in U_F$ let $\rho_{GF}\colon\cA_F\to\cA_G$ denote the restriction. Since $P^n\in U_F$ for all $F$, the nonempty open stratum $V_{P^n,Q}$ lies in every $V_{F,Q}$, and we fix a base point $x_0\in V_{P^n,Q}$. A sheaf on $\cF_\preceq$ is determined by its values on the minimal opens $U_F$, and its stalk at $F$ is its value on $U_F$. Hence a morphism of sheaves is a quasi-isomorphism if and only if it is one on every $U_F$.

\subsection{Local cohomology}

The key input is the analogue of Lunts's statement that the restriction of $\Omega$ to every $V_i$ generates $H^\ast(V_i)$. In our setting the generators have degree $4$.

\begin{lemma}\label{lem:local-cohomology}
Let $F\in\cF(P^n)$ have codimension $k$, and let $H_F:=\lambda^b(F)\cong Q^k$ with inclusion $\phi\colon H_F\hookrightarrow Q^n$. Write $\phi^*:=H^\ast(B\phi)\colon H^\ast(BQ^n)\to H^\ast(BH_F)$. Then

\begin{equation} \label{eq:coh-stalk}
H^\ast(\cA_F)\cong H^\ast(BH_F;\R)=\R[u_1',\dots,u_k']
\end{equation} with $|u_j'|=4$, and under this isomorphism $[\Omega_i|_{V_{F,Q}}]=\phi^*(u_i)$ Moreover, $H^m(\cA_F)=0$ for $m=1,2,3$.
\end{lemma}

\begin{proof}
Fix $f\in\mathring F$, so that $O_f\cong Q^n/H_F$, and let $i\colon EQ^n\times_{Q^n}O_f\cong EQ^n/H_F\hookrightarrow V_{F,Q}$ be the inclusion. By Proposition~\ref{prop:borel-deformation-retract}, $V_{F,Q}$ deformation retracts onto $EQ^n\times_{Q^n}O_f$, so $i$ is a homotopy equivalence. Since $EQ^n$ is contractible and $H_F$ acts freely on it, $EQ^n/H_F\simeq BH_F$. The classifying map $\pi[e,y]=[e]$ restricts along $i$ to the projection $EQ^n/H_F\to EQ^n/Q^n=BQ^n$, which is a model for $B\phi$. The isotropy group of a point of $O_f$ is a conjugate of $H_F$ in $Q^n$. Conjugation acts trivially on $H^\ast(BQ^n)$, so $\phi^*$ does not depend on this choice. The de Rham theorem now gives \ref{eq:coh-stalk}. Since $BH_F$ is connected and its cohomology is concentrated in degrees divisible by $4$ yields the triviality of $H^m(\cA_F)$ for  $m=1,2,3$.
\end{proof}

\begin{proposition}\label{prop:generation}
For every $F\in\cF(P^n)$ the classes $[\Omega_i|_{V_{F,Q}}]$ generate $H^\ast(V_{F,Q})$ as a ring. Equivalently, the map
\begin{equation*}
H^{\ast}(BQ^n;\R)\to H^\ast(\cA_F),\qquad u_i\mapsto[\Omega_i|_{V_{F,Q}}],
\end{equation*}
is surjective, and its kernel is the ideal $(I_F)$ generated by the linear subspace 
\begin{equation*}
    I_F:=\ker\left(H^{\ast}(BQ^n;\R) \to H^\ast(\cA_F)\right)\cap W. 
    \end{equation*}
\end{proposition}

\begin{proof}
Recall from Lemma~\ref{lem:local-cohomology} that  $\phi^*
\colon H^\ast(BQ^n)\to H^\ast(BH_F)$. By \cite[Prop.~2.1.7]{hopkinson2012quoric}, the subgroup $H_F\cong Q^k$ of $Q^n$ is of the form $Q(u^1)\cdots Q(u^k)$, where $Q(u^j)=\{(\psi_{u^j_i}(t))_{i}: t\in Q\}$, each $u^j_i$ is $0$ or a unit quaternion, and for each $i$ at most one $u^j_i$ is nonzero. Set $A_j:=\{i:u^j_i\neq0\}$. These sets are nonempty and pairwise disjoint (they are the characteristic sets of \cite[Def.~2.1.3]{hopkinson2012quoric}), and by \cite[Prop.~2.1.8]{hopkinson2012quoric} they depend only on the conjugacy class of $H_F$. Every endomorphism of $Q$ is trivial or inner \cite[Prop.~2.0.3]{hopkinson2012quoric}, and inner automorphisms act trivially on $H^\ast(BQ)$. Hence
\begin{equation*}
\phi^*(u_i)=\begin{cases}u_j' & i\in A_j,\\ 0 & i\notin\bigcup_jA_j.\end{cases}
\end{equation*}
Since every $A_j$ is nonempty, each generator $u_j'$ of $H^\ast(BH_F)$ is the image of some $u_i$, so $\phi^*$ is surjective. Its kernel is the ideal generated by the linear forms $u_i$ with $i\notin\bigcup_jA_j$ and $u_i-u_{i'}$ with $i,i'\in A_j$ for some $j$. The span of these forms is $\ker(\phi^*)\cap W=I_F$, so $\ker(\phi^*)=(I_F)$.
\end{proof}

\begin{corollary}\label{cor:cohomology-sheaf}
If $G\in U_F$ then $I_F\subseteq I_G$. The cohomology sheaf of $\cA$ is the sheaf $\cH$ with $\cH(U_F)=H^{\ast}(BQ^n;\R)/(I_F)$, zero differential, and restrictions the quotient maps

\begin{equation*}
H^{\ast}(BQ^n;\R)/(I_F)\to H^{\ast}(BQ^n;\R)/(I_G).
\end{equation*}
\end{corollary}

\begin{proof}
Since the $\Omega_i$ are global forms, the map $S\to H^\ast(\cA_G)$ is the composite of $H^{\ast}(BQ^n;\R)\to H^\ast(\cA_F)$ with $\rho_{GF}^*$. Hence
\begin{equation*}
\ker(H^{\ast}(BQ^n;\R)\to H^\ast(\cA_F))\subseteq\ker(H^{\ast}(BQ^n;\R)\to H^\ast(\cA_G)),
\end{equation*}and intersecting with $W$ gives $I_F\subseteq I_G$. The stalk of $H^\ast(\cA)$ at $F$ is
\begin{equation*}
H^\ast(\cA_F)\cong H^{\ast}(BQ^n;\R)/(I_F)
\end{equation*} 
by Proposition~\ref{prop:generation}, and the restriction maps are induced by $\rho_{GF}$. By the factorisation above they are the quotient maps.
\end{proof}

\subsection{The sheaf $\cB$ and its cohomology}

Let $E_F:=\Omega|_{V_{F,Q}}=\mathrm{span}\{\Omega_i|_{V_{F,Q}}\}\subseteq\cA_F^4$ and define a graded subspace $C_F\subseteq\cA_F$ by
\begin{gather*}
C_F^0=\{f\in\cA_F^0: f(x_0)=0\},\qquad C_F^1=\cA_F^1,\qquad C_F^2=\cA_F^2,\\
C_F^3=\{\kappa\in\cA_F^3: d\kappa\in E_F\},\qquad C_F^4=E_F,\qquad C_F^m=0\ \text{ otherwise}.
\end{gather*}
Since $dE_F=0$, $C_F$ is a subcomplex of $\cA_F$. Moreover $\rho_{GF}(C_F)\subseteq C_G$, because $\rho_{GF}$ commutes with $d$, maps $\Omega_i|_{V_{F,Q}}$ to $\Omega_i|_{V_{G,Q}}$, and $x_0\in V_{G,Q}$. Put
\begin{equation*}
\cB(U_F):=\mathrm{Sym}(C_F),
\end{equation*}
the free graded-commutative algebra on $C_F$, with the differential induced by $d$. The maps $\mathrm{Sym}(\rho_{GF}|_{C_F})$ are compatible with composition, so $\cB$ is a sheaf of commutative DG-algebras on $\cF_\preceq$.

\begin{lemma}\label{lem:cohomology-CF}
For every $F\in\cF(P^n)$ the complex $C_F$ has cohomology concentrated in degree $4$, and
\begin{equation*}
H^\ast(C_F)=E_F/dC_F^3\cong W/I_F.
\end{equation*}
These identifications are compatible with the restriction maps $C_F\to C_G$ for $G\in U_F$, which induce the quotient maps $W/I_F\to W/I_G$.
\end{lemma}

\begin{proof}
By \eqref{eq:coh-stalk}, $H^0(\cA_F)=\R$, so the closed elements of $\cA_F^0$ are the constants. They meet $C_F^0$ trivially, hence $H^0(C_F)=0$.

Let $\alpha\in C_F^1=\cA_F^1$ be closed. Since $H^1(\cA_F)=0$ by Lemma~\ref{lem:local-cohomology}, $\alpha=df$ for some $f\in\cA_F^0$. Then $f-f(x_0)\in C_F^0$ and $d(f-f(x_0))=\alpha$, so $H^1(C_F)=0$.

Let $\alpha\in C_F^2=\cA_F^2$ be closed. Since $H^2(\cA_F)=0$ and $C_F^1=\cA_F^1$, we get $\alpha=d\beta$ with $\beta\in C_F^1$, so $H^2(C_F)=0$.

Every closed $3$-form lies in $C_F^3$, because $d\kappa=0\in E_F$. Since $H^3(\cA_F)=0$ and $C_F^2=\cA_F^2$, we get $H^3(C_F)=0$ as well.

In degree $4$ every element of $C_F^4=E_F$ is closed, so $H^4(C_F)=E_F/dC_F^3$, and $dC_F^3=E_F\cap d\cA_F^3$. Therefore $E_F/dC_F^3\to H^4(\cA_F)$ is injective, with image the span of the classes $[\Omega_i|_{V_{F,Q}}]$. The map $W\to E_F/dC_F^3$, $u_i\mapsto[\Omega_i|_{V_{F,Q}}]$, is surjective, and its kernel is $\{\ell\in W:[\ell(\Omega)]=0\in H^4(\cA_F)\}=I_F$ by Proposition~\ref{prop:generation}. This proves the first claim.

For the compatibility, $\rho_{GF}(dC_F^3)\subseteq dC_G^3$, so $\rho_{GF}$ induces a map $E_F/dC_F^3\to E_G/dC_G^3$. Under the identifications with $W/I_F$ and $W/I_G$ it sends $[u_i]$ to $[u_i]$, that is, it is the quotient map, which is well defined since $I_F\subseteq I_G$ by Corollary~\ref{cor:cohomology-sheaf}.
\end{proof}

\begin{proposition}\label{prop:cohomology-B}
For every $F\in\cF(P^n)$,
\begin{equation*}
H^\ast(\cB(U_F))\cong\mathrm{Sym}(W/I_F)=H^{\ast}(BQ^n;\R)/(I_F),
\end{equation*}
compatibly with the restriction maps. In other words, the cohomology sheaf of $\cB$ is the sheaf $\cH$ of Corollary~\ref{cor:cohomology-sheaf}.
\end{proposition}

\begin{proof}
Over a field of characteristic $0$, $H^\ast(\mathrm{Sym}C)\cong\mathrm{Sym}H^\ast(C)$ for every complex of vector spaces $C$ (see \cite[V.3.6, Lemma~(7)]{GeMa}), naturally in $C$. By Lemma~\ref{lem:cohomology-CF}, $H^\ast(\cB(U_F))\cong\mathrm{Sym}(W/I_F)$, and since $H^{\ast}(BQ^n;\R)=\mathrm{Sym}(W)$ this equals $H^{\ast}(BQ^n;\R)/(I_F)$. Naturality and the compatibility in Lemma~\ref{lem:cohomology-CF} show that the restriction maps are the quotient maps $H^{\ast}(BQ^n;\R)/(I_F)\to H^{\ast}(BQ^n;\R)/(I_G)$.
\end{proof}

\subsection{Formality}

\begin{theorem}\label{prop:formality}
There is a zigzag of quasi-isomorphisms of sheaves of commutative DG-algebras
\begin{equation*}
\cA\xleftarrow{\,f\,}\cB\xrightarrow{\,g\,}\cH,
\end{equation*}
where $\cH$ is the cohomology sheaf of $\cA$ with zero differential. In particular $\cA$ is formal.
\end{theorem}

\begin{proof}
\emph{The map $f$.} Let $f_F\colon\mathrm{Sym}(C_F)\to\cA_F$ be the algebra map extending the inclusion $C_F\subseteq\cA_F$. It is a morphism of DG-algebras and commutes with restriction, so it defines a morphism of sheaves $f\colon\cB\to\cA$. By Proposition~\ref{prop:cohomology-B}, $H^\ast(\cB(U_F))\cong H^{\ast}(BQ^n;\R)/(I_F)$. Under this identification $H^\ast(f_F)$ is the map $H^{\ast}(BQ^n;\R)/(I_F)\to H^\ast(\cA_F)$ induced by $u_i\mapsto[\Omega_i|_{V_{F,Q}}]$, because it is multiplicative and agrees with this map on the generators $W/I_F$. It is an isomorphism by Proposition~\ref{prop:generation}. So $f_F$ is a quasi-isomorphism.

\emph{The map $g$.} Let $q_F\colon C_F\to E_F/dC_F^3\cong W/I_F$ be the projection in degree $4$ and zero in the other degrees, where the target has zero differential. It is a morphism of complexes, since $d(C_F^3)=dC_F^3$ maps to $0$, and it is a quasi-isomorphism by Lemma~\ref{lem:cohomology-CF}. Let
\begin{equation*}
g_F:=\mathrm{Sym}(q_F)\colon\cB(U_F)\to\mathrm{Sym}(W/I_F)=H^{\ast}(BQ^n;\R)/(I_F)=\cH(U_F).
\end{equation*}
This is a morphism of DG-algebras. It is a quasi-isomorphism, since $\mathrm{Sym}$ preserves quasi-isomorphisms of complexes in characteristic $0$, again by \cite[V.3.6]{GeMa}. The maps $q_F$ commute with restriction by Lemma~\ref{lem:cohomology-CF}, so the $g_F$ define a morphism of sheaves $g\colon\cB\to\cH$.

Since $f_F$ and $g_F$ are quasi-isomorphisms for every $F$, so are $f$ and $g$. By Corollary~\ref{cor:cohomology-sheaf}, $\cH$ is the cohomology sheaf of $\cA$ with zero differential. Hence $\cA$ is formal. Tensoring with $\C$ gives the statement over $\C$.
\end{proof}

Since the sheaf of DG-algebras $\cA$ is formal, the equivalence of Theorem~\ref{thm:main-theorem} can be expressed purely in terms of the cohomology sheaf $H(\cA)$. 

\begin{corollary}\label{cor:formal}
Let $M^{4n}$ be a quoric manifold. Since $\cA$ is formal, there is an equivalence of triangulated categories
\begin{equation*}
D^{b}_{Q^n,c}(M^{4n}) \xrightarrow{\ \sim\ } D_{H(\cA),c},
\end{equation*}
where $H(\cA)$ is the cohomology sheaf of $\cA$, regarded as a sheaf of DG-algebras with zero differential.
\end{corollary}

\bibliography{references}

\end{document}